\documentclass[a4paper,12pt]{amsart}

\usepackage{amssymb,amsbsy,amsmath,amsfonts,amssymb,amscd}
\usepackage{latexsym}

\input xy
\xyoption{all}

\newcommand\sB{{\mathcal B}}
\newcommand\sN{{\mathcal N}}

\newcommand\Ga{\Gamma}

\newcommand\ga{\gamma}

\newcommand\fie{\varphi}

\newcommand{\CC}{\ensuremath{\mathbb{C}}}

\newcommand{\ZZ}{\ensuremath{\mathbb{Z}}}
\newcommand{\QQ}{\ensuremath{\mathbb{Q}}}

\newcommand{\NN}{\ensuremath{\mathbb{N}}}
\newcommand{\hol}{\ensuremath{\mathcal{O}}}

\newcommand{\HH}{\ensuremath{\mathbb{H}}}
\newcommand{\PP}{\ensuremath{\mathbb{P}}}

\newcommand{\ra}{\ensuremath{\rightarrow}}

\def\eea{\end{eqnarray*}}
\def\bea{\begin{eqnarray*}}

\newcommand\dual{\mathrel{\raise3pt\hbox{$\underline{\mathrm{\thinspace d
\thinspace}}$}}}
\newcommand\qe{\ifhmode\unskip\nobreak\fi\quad $\Box$}       % box for QED

\def\BOX{\hfill\lower.5\baselineskip\hbox{$\Box$}}
\newtheorem{theo}{Theorem}[section]
\newtheorem{remarkk}[theo]{Remark}
\newenvironment{rem}{\begin{remarkk}\rm}{\end{remarkk}}

\newtheorem{defin}[theo]{Definition}
\newenvironment{definition}{\begin{defin}\rm}{\end{defin}}

\newtheorem{prop}[theo] {Proposition}
\newtheorem{question}[theo] {Question}
\newtheorem{cor}[theo]{Corollary}
\newtheorem{lemma}[theo]{Lemma}
\newtheorem{example}[theo]{Example}

\newtheorem{conj}[theo]{Conjecture}
\newtheorem{problem}[theo]{Problem}
\newtheorem{rema}{Remark}[section]

\newcommand{\BT}{\ensuremath{\mathbb{T}}}

\DeclareMathOperator{\Alb}{Alb}
\DeclareMathOperator{\Aut}{Aut}
\DeclareMathOperator{\Sing}{Sing}
\DeclareMathOperator{\Tors}{Tors}
\DeclareMathOperator{\Sym}{Sym}
\newcommand{\Proof}{{\it Proof. }}
\begin{document}

\title[Quotients of products of curves]
{Quotients of  products of curves, new surfaces with $p_g = 0$ and their
fundamental groups.}
\author{I. Bauer, F. Catanese, F. Grunewald , R. Pignatelli}

\thanks{After the few initials examples, constructed 'by hands' by 
the first two named
authors, were presented at a workshop in Pisa in may 2006, the present
work became a central project  of the DFG Forschergruppe 790
"Classification of algebraic surfaces and compact complex manifolds":
  in particular  the
visit of the fourth author to Bayreuth was supported by the DFG.
Complete results were
presented in february 2008 at Warwick University, then at Tokyo University,
and at the Centro De Giorgi, Scuola Normale di Pisa; we thank these
institutions  for their hospitality.}

\date{\today}

\begin{abstract}
We construct many new surfaces of general type with $q=p_g = 0$
whose canonical model is the quotient of the product of two
curves by the action of a finite group $G$,
  constructing  in this way many new interesting  fundamental groups
  which distinguish connected components of the
moduli space of surfaces of general type.

We indeed classify all such surfaces
whose canonical model is singular (the smooth case was classified  in 
an earlier work).

As an important tool we prove a structure theorem
giving a precise description of the fundamental group of
  quotients of products of curves by the action of a finite group $G$.

\end{abstract}

\maketitle

\tableofcontents

%%%%%%%%%%%%%%%%%%%%%%%%%%%%%%%%%%%%%%%%%%%%%%%%%%%%%%
\section*{Introduction}
%%%%%%%%%%%%%%%%%%%%%%%%%%%%%%%%%%%%%%%%%%%%%%%%%%%%%%

The first main purpose of this paper is to contribute to the existing
knowledge about the complex projective surfaces $S$ of general type with
$p_g(S) = 0$ and their moduli spaces, constructing 19 new families
of such surfaces with hitherto unknown fundamental groups.

Minimal surfaces of general type with $p_g(S) = 0$ are known to have invariants
$p_g(S)  =  q (S)= 0, 1 \leq K_S^2 \leq 9$, and to yield a finite
number of irreducible components of the moduli space of surfaces of 
general type.

They represent for algebraic geometers a very difficult test case about
the possibility of extending the Enriques  classification of special surfaces
to surfaces of general type.

They are also very interesting in view of the Bloch conjecture (\cite{bloch}),
predicting that for surfaces with $p_g(S)  =  q (S)= 0$
the group of zero cycles modulo rational equivalence
is  isomorphic to $\ZZ$.

In this paper, using the
beautiful results
of Kimura (\cite{kimura}, see also \cite{gp}), the present results,
and those of the previous paper \cite{bcg}, we  produce
more than 40 families  of  surfaces for which
Bloch's conjecture holds.

Surfaces with $p_g(S)  =  q (S)= 0$ have a very old history, dating back to
1896 (\cite{enr96}, see also \cite{enrMS},
I, page 294) when Enriques constructed the
so called Enriques surfaces in order to give a counterexample to the conjecture
of  Max Noether that any such surface should be rational.

The first surfaces of general type with $p_g = q =
0$ were constructed in the 1930' s by   Luigi Campedelli and Lucien
Godeaux (cf. \cite{Cam},
\cite{god}): in their honour  minimal surfaces  of general type with
$K_S^2 = 1$  are called
numerical Godeaux surfaces, and those with
$K_S^2 = 2$ are called numerical  Campedelli surfaces.

In the 1970's there was  a big revival of interest in the construction of these
surfaces and in a possible classification.

After rediscoveries of these and other old examples a few new
ones were found through
the efforts of several authors, in particular Rebecca Barlow 
(\cite{barlow}) found
a simply connected numerical Godeaux surface, which played a decisive role
in the study of the differential topology of algebraic surfaces and 4-manifolds
(and also in the discovery of K\"ahler Einstein metrics of opposite sign
on the same manifold, see \cite{clb}).

   A (relatively short) list of the existing  examples appeared
   in the book \cite{bpv}, (see   \cite{bpv}, Vii, $10$ and
references therein,
and see also
\cite{bhpv} for an updated longer list).

There has been  recently  important progress
on the topic,
and the current situation is as follows:

\begin{itemize}
\item
$K_S^2 = 9$:  these surfaces have the unit  ball in $\CC^2$ as universal cover,
and their fundamental group is an arithmetic subgroup  $\Ga$ of $ SU (2,1)$.

This case seems to be completely classified through exciting
new work of Prasad and Yeung and of Steger and Cartright
(\cite{p-y}, \cite{p-yadd}) asserting that  the moduli space  consists
exactly of 50 pairs of complex conjugate surfaces.
\item
    $K_S^2 = 8$: we  pose the  question  whether  in this case the universal
cover must
be the bidisk in $\CC^2$.

Assuming this, a complete classification should be possible.

The classification has already been accomplished  in
\cite{bcg}
for the reducible case where there is a finite \'etale cover which is
isomorphic
to a product of curves: in this case there are exactly 17
irreducible connected components
of the moduli space.

There are many examples, due to Kuga and Shavel (\cite{kug}, \cite{shav})
for the irreducible case, which yield (as in the case $K_S^2 = 9$)
rigid surfaces;
but a complete classification of this second case is still missing.
\end{itemize}

The constructions of minimal surfaces of general type
with $p_g=0$ and with $K^2_S \leq 7$ avalaible in the literature (to the
best of the authors' knowledge, and excluding the results of this 
paper)  are listed in table \ref{tabknown}.

We proceed to a description, with the aim of putting the results of
the present paper in proper perspective.

\begin{table}
\caption{Minimal surfaces of general type with
   $p_g=0$ and $K^2 \leq 7$ avalaible in the literature}
\label{tabknown}
\begin{tabular}[ht]{|c|c|c|c|l|}
\hline
$K^2$  & $\pi_1$& $\pi_1^{alg}$ & $H_1$& References \\
\hline
\hline
1  &$\ZZ_5$&$\ZZ_5$ & $\ZZ_5$&\cite{godold}\cite{tokyo}\cite{miyaokagod}\\
   &$\ZZ_4$&$\ZZ_4$ & 
$\ZZ_4$&\cite{tokyo}\cite{op}\cite{barlow2}\cite{naie94}\\
   &?& $\ZZ_3$ &  $\ZZ_3$&\cite{tokyo}\\
   &$\ZZ_2$&$\ZZ_2$ & $\ZZ_2$&\cite{barlow2}\cite{inoue}\\
   &?&$\ZZ_2$ & $\ZZ_2$&\cite{werner}\cite{wer97}\\
   &$\{1\}$& $\{1\}$&  $\{0\}$ &\cite{barlow}\cite{lp} \\
   & ? & $\{0\}$ &  $\{0\}$& \cite{cg}\cite{dw}  \\
\hline
\hline
2  &$\ZZ_9$&$\ZZ_9$ & $\ZZ_9$& \cite{mlp} \\
    &$\ZZ_3^2$&$\ZZ_3^2$&$\ZZ_3^2$ &\cite{xiao}\cite{mlp} \\
 
&$\ZZ_2^3$&$\ZZ_2^3$&$\ZZ_2^3$&\cite{Cam}\cite{MilesCamp}\cite{peterscamp}\cite{inoue}\\
    &&&&\cite{naie94}\\
    &$\ZZ_2 \times \ZZ_4$&$\ZZ_2 \times \ZZ_4$&$\ZZ_2 \times \ZZ_4$  & 
\cite{MilesCamp}\cite{naie94}\cite{keum} \\
    &$\ZZ_8$&$\ZZ_8$& $\ZZ_8$ & \cite{MilesCamp} \\
    &$Q_8$&$Q_8$ & $\ZZ_2^2$ &\cite{MilesCamp} \cite{beauville96}\\
    &$\ZZ_7$&$\ZZ_7$& $\ZZ_7$ & \cite{cvg} \\
    &?&$\ZZ_6$& $\ZZ_6$ & \cite{np} \\
    &$\ZZ_5$&$\ZZ_5$ & $\ZZ_5$ & \cite{Babbage}\cite{sup} \\
&$\ZZ_2^2$&$\ZZ_2^2$ & $\ZZ_2^2$ & \cite{inoue}\cite{keum} \\
   &?&$\ZZ_3$ & $\ZZ_3$ & \cite{lp2} \\
   &?&$\ZZ_2$ & $\ZZ_2$ & \cite{lp2} \\
    &$\{1\}$& $\{1\}$ & $\{0\}$ & \cite{lp} \\
\hline
\hline
3  &$\ZZ_2^2 \times \ZZ_4$&$\ZZ_2^2 \times \ZZ_4$ & $\ZZ_2^2 \times 
\ZZ_4$& \cite{naie94} \cite{keum} \cite{mlp3} \\
    &$Q_8 \times \ZZ_2$&$Q_8 \times \ZZ_2$ & $\ZZ_2^3$ & 
\cite{burniat}\cite{peters} \cite{inoue}\\
    &? &? & $\ZZ_2$ & \cite{pps3b} \\
    &$\{1\}$&$\{1\}$ & $\{0\}$ &\cite{pps3} \\

\hline
\hline
4  &$1 \rightarrow \ZZ^4 \rightarrow \pi_1 \rightarrow
\ZZ_2^2\rightarrow 1$&$\hat{\pi}_1$& $\ZZ_2^3 \times \ZZ_4$ & 
\cite{naie94}\cite{keum}\\
    &$Q_8 \times \ZZ_2^2$ &$Q_8 \times \ZZ_2^2$ & $\ZZ_2^4$ 
&\cite{burniat}\cite{peters}\cite{inoue}\\
    &$\{1\}$&$\{1\}$ & $\{0\}$ & \cite{pps4} \\
\hline
\hline
5   &$Q_8 \times \ZZ_2^3$&$Q_8 \times \ZZ_2^3$ & $\ZZ_2^5$ & 
\cite{burniat}\cite{peters}\cite{inoue}\\
    &?&?&?& \cite{inoue}\\
\hline
\hline
6   &$1 \rightarrow \ZZ^6 \rightarrow \pi_1 \rightarrow
\ZZ_2^3\rightarrow 1$&$\hat{\pi}_1$& $\ZZ_2^6$ & 
\cite{burniat}\cite{peters}\cite{inoue}\\
    &?&?&?& \cite{inoue}\cite{mlp4b}\\
\hline
\hline
7   &$1 \rightarrow \Pi_3 \times \ZZ^4 \rightarrow \pi_1 \rightarrow
\ZZ_2^3\rightarrow 1$&$\hat{\pi}_1$&?& \cite{inoue}\cite{mlp01} 
\cite{toappear}\\
\hline
\end{tabular}
\end{table}

\begin{itemize}
\item
$K_S^2 = 1$, i.e., numerical Godeaux surfaces: it is conjectured by Miles Reid
that the moduli space should have exactly five irreducible connected
components,
distinguished by the fundamental group, which should be  a cyclic group
$\ZZ_m$ of order
$ 1 \leq m \leq 5$ (\cite{tokyo} settled the case where the order $m$ of
the first homology
  group is at least 3; \cite{barlow}, \cite{barlow2} and
\cite{werner}
   were the first to show the occurrence of the two other groups).

\item
$K_S^2 = 2$, i.e., numerical Campedelli surfaces: here, it is  known
that the order of the algebraic fundamental group is at most $9$,
and the cases of order $8,9$ have been classified by Mendes Lopes,
    Pardini and Reid (\cite{mlp}, \cite{mlpr}, \cite{MilesCamp}), who
    show in particular that the fundamental group equals the algebraic
    fundamental group and cannot be a dihedral group $D_4$ of order $8$.
Naie (\cite{naie}) showed that the group $D_3$ of order $6$
cannot occur as the fundamental group of a numerical Campedelli surface.
By the work of  Lee and Park (\cite{lp}), one knows that there exist
simply connected numerical Campedelli surfaces.

Our first result here
is the construction of two families of
    numerical Campedelli surfaces with fundamental
group
$\ZZ_3$. Recently Neves and Papadakis  (\cite{np}) constructed a
numerical Campedelli surface with algebraic fundamental group $\ZZ_3$,
while Lee and Park (\cite{lp2}) constructed one with algebraic
fundamental group $\ZZ_2$, and one with algebraic
fundamental group $\ZZ_3$ was added in the second version of the same paper.

Open conjectures are:

\begin{conj}
Is the fundamental group $\pi_1(S)$ of a numerical Campedelli
surface finite?
\end{conj}
\begin{question}\label{qu1}
Does every group of order $\leq 9$ except $D_4$ and $D_3$ occur?
\end{question}
The answer to question \ref{qu1} is completely open for $\ZZ_4$;  for
$\ZZ_6, \ZZ_2$ one suspects that these fundamental groups
are realized by the the  Neves-Papadakis surfaces,
respectively by  the  Lee-Park surfaces.

Note that the existence of the case where $\pi_1(S) = \ZZ_7$ is shown 
in the paper \cite{cvg}
(where the result is slightly hidden).

\item
$K_S^2 = 3$: here there were two examples of nontrivial fundamental groups,
the first one due to Burniat and  Inoue,  the second one to Keum and 
Naie (\cite{burniat},
\cite{inoue}, \cite{keum}
\cite{naie94}).

Mendes Lopes and Pardini proved (\cite{mlp1})  that for $K_S^2 = 3$
the algebraic
fundamental group is finite, and one can ask as in 1) above whether
also $\pi_1(S)$ is
finite. Park, Park and Shin (\cite{pps3}) showed the existence of
simply connected
surfaces, and of surfaces with
torsion $\ZZ_2$ (\cite{pps3b}).

Other constructions  were given in \cite{sbc},  together
with two more examples
   with $p_g(S)=0, K^2 = 4,5$: these turn out however to be the same 
as the Burniat surfaces.
\item
$K_S^2 = 4$: there were known up to now three examples of fundamental groups,
the trivial one (Park, Park and Shin , \cite{pps4}), a finite one, and an
infinite one.
We show here the existence of 7 new groups,
3 finite and 4 infinite: thus minimal surfaces
with $K_S^2 =
4$,
$p_g(S) = q(S)=0$ realize at least 10 distinct topological types.
\item
$K_S^2 = 5,6,7$: there was known up to now only one example of
fundamental group
for each such value of $K_S^2 $
\item
$K_S^2 = 6$ : we show in this paper the existence of 6 new groups,
three of which
finite: thus minimal surfaces with $K_S^2 = 6$,
$p_g(S) = q(S)=0$ realize at least 7 distinct topological types.
\item
$K_S^2 = 7$ : we shall show elsewhere (\cite{toappear}) that the 
fundamental group of these surfaces,
constructed by Inoue in \cite{inoue},
have a fundamental group fitting into an exact sequence
$$1 \rightarrow \Pi_3 \times \ZZ^4 \rightarrow \pi_1 \rightarrow
\ZZ_2^3\rightarrow 1.$$
This motivates the following further question
\end{itemize}

\begin{question}  Is it true that fundamental groups of surfaces of 
general type
with $ q=p_g=0$ are finite for $ K^2_S \leq 3$, and infinite for $ 
K^2_S \geq 7$?
\end{question}

Let us observe that, as  we just saw, for $ 4 \leq K^2_S \leq 6$
both finite and infinite groups occcur; moreover, there is no reason 
to exclude the existence
of such   surfaces with $ K^2_S = 7$ and with finite fundamental group.

In a forthcoming paper (\cite{bp}) the first and the fourth  author 
shall show the existence, for $ K^2_S = 5$, of
7 new fundamental groups, 3 of which infinite, and also
of 3 new finite fundamental groups for $ K^2_S = 4$, and
of 4 new finite fundamental groups for $ K^2_S = 3$.

One of the reasons why we succeed in constructing so many  new families of
surfaces with $K_S^2 = 2,4, 6$, thereby showing that 14 new fundamental groups
are realized, is the fact that we use a systematic method
and classify completely the following situation.

We consider the algebraic surfaces whose canonical models arise
    as quotients $X =(C_1 \times
C_2)/G$   of the product $C_1 \times
C_2$ of two curves of  genera
$g_1 := g(C_1), g_2 : = g(C_2) \geq 2$, by the action of a finite group $G$.
In other words, we make the restriction that $X$ has only rational
double points as
singularities.

We achieve a complete classification of  the surfaces $X$ as above which have
$p_g(X) = q(X)=0$.

An interesting corollary of our classification is
that then either

i) $G$ acts freely, i.e., equivalently, $X$
is smooth (hence $K^2_S = 8$, and we are in the case previously classified
in \cite{bcg}), or

ii) $X$ has only nodes ($A_1$ singularities) as singular points:
in this case moreover  their number is even and equal to $8 - K^2_S$.

This result explains why our construction only produces
new surfaces with $K^2_S = 2,4,6$, as illustrated by the
following  three main theorems concerning surfaces with $p_g(S) = q(S)=0$.

For convenience of notation we shall use the following

\begin{defin}
A surface with rational double points which is the quotient of a
product of curves
by the action of a finite group will be called a {\em 
product-quotient} surface.

Number 2 will mean: ` number 2 in the list given  in Table \ref{surfaces}'.
\end{defin}

\begin{theo}\label{campedelli}
There exist eight  families of product-quotient  surfaces yielding 
numerical Campedelli surfaces (i.e.,
minimal surfaces with
$K^2_S = 2, p_g (S)=0$).

Two of them (numbers 2 and 5,
with group
$G=\mathfrak{S}_5$, resp. $G=\mathfrak{S}_3 \times \mathfrak{S}_3$) 
yield fundamental group $ \ZZ_3$.
\end{theo}

Our classification also shows the existence  of a family of product
-quotient  surfaces
yielding numerical Campedelli surfaces  with fundamental group $ \ZZ_5$
(but numerical Campedelli surfaces  with fundamental group $ \ZZ_5$
had already been constructed in \cite{Babbage}),
respectively with
fundamental group
$\ZZ_2^2$ (but such fundamental group already appeared in \cite{inoue}).

\begin{theo}\label{4}
There exist precisely eleven families  of product-quotient  surfaces yielding
minimal surfaces with
$K^2_S = 4, p_g (S)=0$.

Three of these families (numbers 9, 12, 15) realize
3 new finite fundamental groups, $\ZZ_{15}$, $G(32,2)$ (see rem. 
\ref{---}) and $( \ZZ_3)^3$.

Six of these families (numbers 8 and 14, 10 and 13, 11, 16) realize
4 new infinite fundamental groups.

\end{theo}

The two families 17, 18 realize instead a fundamental group $\Ga$ which is
isomorphic to the one of the surfaces constructed by Keum
(\cite{keum}) and Naie
(\cite{naie94}). It was not clear to us whether these four families 
of surfaces all
belonged to a unique
irreducible component of the moduli space.
This question has motivated (in the time between the first and the 
final version of this paper)
work by the first two authors (\cite{keumnaie}): it is shown there a 
construction
of Keum-Naie surfaces yielding a unique irreducible connected component of
the moduli space, to which all
minimal  surfaces $S$
with fundamental group $\Ga$, $K^2_S = 4, p_g (S)=0$ belong, provided either

a) they are homotopically equivalent to a Keum-Naie surface, or

b) they admit a deformation to a surface with ample canonical divisor.

We come here to an important difference between the case
of surfaces isogenous to a product of curves with  $K^2_S = 8, p_g (S)=0$,
classified in \cite{bcg}, and the
cases classified  the present paper.
Namely, unlike the case of surfaces isogenous to a product of curves,
our product-quotient
surfaces do not necessarily form a connected, or even an irreducible component
of the moduli space.

In fact, by the Enriques-Kuranishi  inequality
for the number of
moduli $ M(S)$ (the local dimension of the moduli space),
$ M (S)  \geq 10 \chi(S) -2 K^2_S$, our families number 1-10 and 14
necessarily do
not contain  an open set  of the moduli space.

It is then an interesting question to see first of all whether the nodes of our
product-quotient surfaces
can be smoothed, or even independently smoothed (see \cite{enr} for the
consequences about singularities of the moduli space),
and more generally to investigate  the local structure of the moduli 
space at the sublocus corresponding to
our families of product-quotient surfaces.

A second harder problem is to try to describe the irreducible (resp.: 
connected)
components of the moduli space containing these subloci.

Still referring to  table \ref{tabknown} shown previously, this is 
possible in some cases.
Note first of all that Mendes Lopes and Pardini (\cite{burniatmp})
proved that the 'primary' Burniat surfaces
(those with
$K^2 = 6$) form a connected component of the moduli space.
This result has been reproven in the meantime by the first two authors in
\cite{burniat1}, showing more generally that any surface 
homotopically equivalent to a primary Burniat
surface is indeed a primary Burniat surface.

The positive results of
\cite{burniat2}, showing that also Burniat surfaces with $K^2_S = 5,4$
form 3 connected components of the moduli space, seem to depend 
heavily on the fact that these
surfaces are realized as finite  Galois covers with group $(\ZZ_2)^2$ 
of Del Pezzo surfaces.

In other words, our constructions offer several interesting 
challenges concerning
the investigation of the moduli space, but these cannot be solved all 
at once, and
for each family of surfaces one has to resort very much to
the special geometric properties which the surfaces possess, and 
which are not being lost
by deformation.

The next case is all the more challenging, since the Enriques-Kuranishi
inequality gives no information.

\begin{theo}\label{6}
There exist eight families (numbers 19-24) of product
-quotient  surfaces yielding minimal surfaces with
$K^2_S = 6, p_g (S)=0$ and realizing 6 new fundamental groups,
three of them finite and three of them infinite.
   In particular, there exist minimal surfaces of general type with
$p_g=0$, $K^2=6$ and
with finite fundamental group.
\end{theo}

Since a main new contribution of our paper is the calculation of
new fundamental groups (as opposed to new algebraic fundamental groups),
our second main purpose is to describe in greater generality
the fundamental groups of smooth projective varieties which
occur as the minimal resolutions of the quotient of a product of
curves by the action
of a finite group.

More precisely,  in this paper we are concerned with the following
rather (but not completely) general situation.

Let $C_1, \ldots , C_n$ be smooth projective curves of respective
genera $g_i \geq 2$ and
suppose that there is a finite group $G$ acting faithfully on each of the $n$
curves.

Then we
consider the diagonal action of $G$ on the Cartesian product $C_1 \times
\ldots , C_n$, and
the quotient $X:= (C_1 \times \ldots \times C_n) / G$.

Recall that if the action
of $G$ on $C_1\times \ldots \times C_n$ is free, then $X$ is said to be a
  {\em variety isogenous to a (higher) product
of curves (of unmixed type)}. These varieties were introduced and
studied by the second
author in \cite{FabIso}, mainly in the case $n=2$. In this case the
universal cover of $X$
is  the product of
$n$ copies of the upper half plane
$\mathbb H \times \ldots \times \mathbb H$.

We  drop here the hypothesis that $G$ acts freely on $C_1 \times
\ldots \times C_n$.
Then $X$ has singularities, but since they are cyclic quotient
singularities, they can be
resolved, in the case where they are isolated singularities, by  a
simple normal
crossing divisor whose components are smooth rational varieties
(\cite{fujiki}). By van
Kampen's theorem this implies that the fundamental group of $X$ is equal to the
fundamental group of a minimal desingularisation $S$ of $X$.

One of the preliminary observations  of \cite{FabIso} was the following:

\begin{prop}
     Let $X:= (C_1 \times \ldots \times C_n) / G$ be isogenous to a
product as above. Then
the fundamental group of $X$ sits in an exact sequence
$$ 1 \rightarrow \Pi_{g_1} \times \ldots \times \Pi_{g_n} \rightarrow \pi_1 (X)
\rightarrow G \rightarrow 1,
$$ where $\Pi_{g_i}:=\pi_1(C_i)$, and this extension is determined by
the associated maps
$ G \ra \mathfrak Map_{g_i} : = Out (\Pi_{g_i})$  to the respective
Teichm\"uller modular
groups.
\end{prop}

If one drops the assumption about  the freeness of the action of $G$ on
$C_1 \times \ldots \times C_n$, there is no reason that the behaviour of the
fundamental group of the
quotient should be similar to the above situation. Nevertheless it
turns out that, as an
abstract group,  the fundamental group admits a very similar
description.

Before giving the main result of the first part of our paper, which is a
structure theorem for the
fundamental group of $X=(C_1 \times \ldots \times C_n) / G$, we need
the following
\begin{defin}
     We shall call the fundamental group $\Pi_{g}:=\pi_1(C)$ of a smooth
compact complex curve
of genus
$g$ a {\em (genus g) surface group}.
\end{defin} Note that we admit also the ``degenerate cases'' $g = 0, 1$.

\begin{theo}\label{pi} Let $C_1, \ldots , C_n$ be compact complex
curves of respective
genera
$g_i \geq 2$ and let $G$ be a finite group acting faithfully on each
$C_i$ as a group of biholomorphic transformations.

  Let $X=(C_1 \times
\ldots \times C_n) /
G$, and denote by $S$ a minimal desingularisation of $X$. Then the fundamental
group
$\pi_1(X) \cong \pi_1(S)$ has a normal subgroup $\mathcal N$ of finite index
which  is isomorphic to the product of surface groups, i.e., there
are natural numbers
$h_1, \ldots , h_n \geq 0$ such that $\mathcal N \cong \Pi_{h_1}
\times \ldots \times
\Pi_{h_n}$.
\end{theo}

\begin{rema}
In the case of dimension $n=2$ there is no loss of generality in assuming
that $G$ acts faithfully on each
$C_i$ (see \cite{FabIso}). In the general case there will be a group
$G_i$, quotient of
$G$, acting faithfully on $C_i$, hence the strategy should slightly be changed
in the general case.
\end{rema}

\begin{rema}
The fundamental groups of certain 3-manifolds which arise as 
quotients of hyperbolic 3-space by
groups which act discontinuously but not necessarily fixpoint freely 
are discussed in \cite{GM12}. Unlike in the case treated the present 
paper, here the groups do not contain CAT(0) groups of finite index.
\end{rema}

We shall now give a short description of the proof of theorem
\ref{pi} in the case $n=2$.
The case where
$n$ is arbitrary is exactly the same.

We need to recall the definition of an orbifold surface group

\begin{defin}
An {\em orbifold surface group} of genus $g'$ and multiplicities
$m_1, \dots m_r$ is the  group  presented as follows:

\begin{multline*}
\mathbb T (g';m_1, \ldots ,m_r) :=
\langle a_1,b_1,\ldots , a_{g'},b_{g'}$,  $c_1,\ldots, c_r |\\
     c_1^{m_1},\ldots ,c_r^{m_r},\prod_{i=1}^{g'} [a_i,b_i] \cdot c_1\cdot
\ldots \cdot c_r \rangle.
\end{multline*}
In the case $g'=0$ it is called a {\em polygonal group}.

The sequence $(g';m_1, \dots m_r)$ is called the signature of the
orbifold surface group.

\end{defin}

The above definition shows that an orbifold fundamental group is the
factor group
of the fundamental group of the complement, in a complex curve $C'$
of genus $g'$,
of a finite set of $r$ points $\{p_1, \dots , p_r \}$, obtained by
dividing modulo the
normal subgroup generated by $\ga_1^{m_1}, \dots , \ga_r^{m_r}$,
where for each $i$
$\ga_i$ is a simple geometric loop starting from the base point and
going once around the
point
$p_i$ counterclockwise (cf. \cite{FabIso}).

Hence,  by Riemann's existence theorem, to give an action of a finite group $G$
on a curve $C$ of genus $g \geq 2$ is equivalent to giving:

     1) the quotient curve $C' := C /G$

2) the branch point set $\{p_1, \dots p_r \} \subset C'$

3) a surjection of the fundamental group  $\pi_1 ( C'
\setminus \{p_1,
\dots p_r \})$ onto $\mathbb T (g';m_1, \ldots ,m_r) $, such that the given
generators of $\mathbb T (g';m_1, \ldots ,m_r) $ are image elements of a
{\em standard basis} of $\pi_1 ( C' \setminus \{p_1,
\dots p_r \})$.

This means that  $a_1,b_1,\ldots ,
a_{g'},b_{g'}$ are image elements of a symplectic basis of the 
fundamental group
of $C'$, while each
$c_i$ is the image of a simple geometric loop around the point $p_i$.

4)  a surjective homomorphism
$$
\varphi \colon \BT (g';m_1, \ldots ,m_r) \rightarrow G,
$$ such that

5) $\fie (c_i) $ is an element of order exactly $m_i$ and

6) {\em Hurwitz' formula}
holds:
$$ 2g - 2 = |G|\left(2g'-2 + \sum_{i=1}^r \left(1 -
\frac{1}{m_i}\right)\right).
$$

Therefore in our situation we have two surjective homomorphisms
$$
\varphi_1 \colon \BT_1 : =  \BT(g'_1;m_1,\ldots, m_r) \rightarrow G,
$$
$$
\varphi_2 \colon  \BT_2 : =  \BT(g'_2;n_1,\ldots, n_s) \rightarrow G.
$$

\noindent

We define the fibre product $\HH : = {\HH}(G;\varphi_1,\varphi_2)$ as
\begin{equation} \HH : = {\HH}(G;\varphi_1,\varphi_2):= \{ \,
(x,y)\in \BT_1\times \BT_2\
|\
\varphi_1(x)=\varphi_2(y)\,\}.
\end{equation}

Then the exact sequence
\begin{equation} 1 \rightarrow \Pi_{g_1} \times \Pi_{g_2} \rightarrow
\BT_1 \times \BT_2
\rightarrow G \times G \rightarrow 1,
\end{equation} where $\Pi_{g_i} := \pi_1(C_i)$, induces an exact sequence

\begin{equation} 1 \rightarrow \Pi_{g_1} \times \Pi_{g_2} \rightarrow
{\HH}(G;\varphi_1,\varphi_2)
\rightarrow G \cong \Delta_G \rightarrow 1.
\end{equation} Here $\Delta_G \subset G \times G$ denotes the
diagonal subgroup.

\begin{definition}
     Let $H$ be a group. Then its {\em torsion subgroup} ${\Tors}(H)$
is the normal
subgroup  generated by all elements of finite order in $H$.
\end{definition}

The first observation is that one can calculate our fundamental groups
via a simple algebraic recipe: $\pi_1((C_1 \times
C_2)/G)
\cong {\HH}(G;\varphi_1,\varphi_2) / {\Tors}(\HH)$.

\medskip\noindent

In this algebraic setup, we need to calculate
${\HH}(G;\varphi_1,\varphi_2) / \Tors(\HH)$. Our proof is rather 
indirect, and we
have no direct
geometric construction leading to our result (even if we can then explain
its geometric meaning).

The strategy is now the following: using the structure of orbifold
surface groups we
construct an exact sequence
$$ 1 \rightarrow E \rightarrow \mathbb{H} / {\Tors}(\mathbb{H}) \rightarrow
\Psi(\hat{\mathbb{H}}) \rightarrow 1,
$$ where
\begin{itemize}
\item[i)] $E$ is finite,
\item[ii)] $\Psi(\hat{\mathbb{H}})$ is a subgroup of finite index in a product
      of orbifold surface groups.
\end{itemize}

Condition
$ii)$ implies that $\Psi(\hat{\mathbb{H}})$ is residually finite and
``good'' according to
the following

\begin{defin}[J.-P. Serre] Let $\mathbb{G}$ be a group, and let
$\tilde{\mathbb{G}}$ be
its profinite completion. Then $\mathbb{G}$ is said to be {\em good}
iff the homomorphism
of cohomology groups
$$ H^k(\tilde{\mathbb{G}},M) \rightarrow H^k(\mathbb{G},M)
$$ is an isomorphism for all $k \in \NN$ and for all finite
$\mathbb{G}$ - modules $M$.
\end{defin}

Then we use the following result due to F. Grunewald, A.
Jaikin-Zapirain, P. Zalesski.

\begin{theo} {\bf(\cite{GZ})} Let $G$ be residually finite and good,
and let $\varphi \colon
H
\rightarrow G$ be surjective with finite kernel. Then $H$ is residually finite.
\end{theo}

The above theorem implies that $\mathbb{H} / {\Tors}(\mathbb{H})$
is residually finite,
whence there is a subgroup $\Gamma \leq \mathbb{H} / {\bf
        Tors}(\mathbb{H})$ of finite index such that
$$
\Gamma \cap E = \{1\}.
$$ Now, $\Psi(\Gamma)$ is a subgroup of $\Psi(\hat{\mathbb{H}})$ of finite
      index, whence of finite index in a product of orbifold surface groups, and
$\Psi| \Gamma$ is injective. This easily implies our result.

In the second part of the paper, as already explained at length,
   we use the above results in order
to study the moduli spaces of surfaces $S$ of general type  with
$p_g(S)=q(S)=0$.

While the investigation of the moduli space of surfaces $S$ of
general type  with
$p_g(S)=q(S)=0$ and $K^2_S \leq 7$ is (as our research indicates) an extremely
difficult task, for
instance since there could be over a hundred  irreducible components, it makes
sense to try first to ask some  'easier' problems:

\begin{problem} Determine the pairs $ ( K^2_S, \pi_1 (S)) $
for all the surfaces with
$p_g(S)=q(S)=0$.
\end{problem}

\smallskip

Note that the pair $ ( K^2_S, \pi_1 (S)) $ is an invariant of the 
connected component of the moduli space,
hence we know that there are only a finite numbers of such pairs.

Several question marks in our previous table rally about the following

\begin{conj} The fundamental groups $ \pi_1 (S) $
of surfaces with
$p_g(S)=q(S)=0$ are residually finite.
\end{conj}

\begin{rem}
There are algebraic surfaces with non residually finite fundamental 
groups, as shown by Toledo
in \cite{Toledo} (see also \cite{trento}), answering thus a question 
attributed to J.P. Serre.

Since in both types of examples one takes  general hyperplane 
sections of varieties of general type,
then necessarily we get surfaces with $p_g > 0$.  Evidence for the 
conjecture is provided by  the
  existing examples, but a certain wishful thinking is also involved 
on our side.

\end{rem}

One could here ask many more questions (compare \cite{milesLNM}),
but  as a beginning step it is important to start providing several examples.

  Here we give a  contribution
to the above
questions, determining the fundamental groups of the surfaces whose 
canonical models
are the quotient of a product of curves by the action of a
finite group.

This follows from the following complete
classification result:

\begin{theo}\label{classpgq=0}
All the surfaces
$X:= (C_1
\times C_2)/G$, where $G$ is a finite group with an unmixed action on a product
$C_1 \times C_2$ of smooth
projective curves $C_1, C_2$ of respective genera $g_1, g_2 \geq 2$ such that:
\begin{itemize}
\item[i)] $X$ has only {\em rational double points} as singularities,
\item[ii)] $p_g(S)=q(S)=0$
\end{itemize}
are obtained by a pair of appropriate epimorphisms of polygonal
groups $\BT_1, \BT_2$ to a
finite group $G$  as listed in table
\ref{surfaces}, for an appropriate choice of respective branch sets in $\PP^1$.
\end{theo}

     Let us briefly illustrate the strategy of proof for the above theorem.
The first step is to
show that, under the above hypotheses, the singularities of
$X$ are either $0,2,4$ or $6$ nodes ($A_1$-singularities). And,
according to the number of
nodes,
$K^2_X = K^2_S
\in
\{8,6,4,2\}$. Since all surfaces isogenous to a product (i.e., with
$0$ nodes) and
with $p_g(S)=q(S)=0$ have
been completely classified in \cite{bcg}, we restrict ourselves
here to the case of
$2$,
$4$, resp. $6$ nodes.

Note furthermore that the assumption $q=0$ implies that $C_i/G$ is
rational, i.e., $g'_1
= g'_2 = 0$.

We know that $X$ determines the following data
\begin{itemize}
     \item a finite group $G$,
\item two polygonal groups $\BT_1:=\BT(0;m_1,\ldots, m_r)$,
$\BT_2:=\BT(0;n_1,\ldots, n_s)$, of respective signatures  $T_1 =
(m_1, \ldots , m_r)$,
$T_2 = (n_1,
\ldots , n_s)$
\item two surjective homomorphisms (preserving the order of the
generators) $\varphi_i \colon
\BT_i
\rightarrow G$, such that the stabilizers fulfill certain conditions
ensuring that $X$
has only $2,4$ or
$6$ nodes respectively.
\end{itemize}

Our second main result,  summarized in table
\ref{surfaces}, contains
moreover the description of the first homology group $ H_1 (S, \ZZ)$
(the abelianization
of
$\pi_1(S)$), and an indication of the number of irreducible families
that we construct in
each case.

Note that, in order to save space in table \ref{surfaces}, we
have used a particular notation for the signatures: for example in the
first row $4^3$ stands for $(4,4,4)$. A more precise explanation of
the table is to be found in subsection \ref{explainthetable}, including the
definition of all listed groups in remark \ref{---}. More details on
the constructed surfaces are in section \ref{thedescription}.

\begin{table}[ht]
\caption{The surfaces}
\label{surfaces}
\begin{tabular}{|c|c|c|c|c|c|c|c|c|}
\hline
$K^2$& $T_1$  &  $T_2$&$g_1$&$g_2$&G&\# fams& $H_1$& $\pi_1(S)$\\
\hline\hline
2&2, 3, 7 &
$4^3$&3&22&PSL(2,7)&2&$\ZZ_2^2$&$\ZZ_2^2$\\
2&2, 4, 5 &2,
$6^2$&4&11&$\mathfrak{S}_5$&1&$\ZZ_3$&$\ZZ_3$\\
2&2, $5^2$&$2^3$,
3&4&6&$\mathfrak{A}_5$&1&$\ZZ_5$&$\ZZ_5$\\
2&2, 4, 6 &$2^3$,
4&3&7&$\mathfrak{S}_4 \times \ZZ_2$&1&$\ZZ_2^2$&$\ZZ_2^2$\\
2&2, $6^2$&$2^3$,3&4&4&$\mathfrak{S}_3 \times
\mathfrak{S}_3$&1&$\ZZ_3$&$\ZZ_3$\\ 2&   $4^3$&
$4^3$&3&3&$\ZZ_4^2$&1&$\ZZ_2^3$&$\ZZ_2^3$\\ 2&$2^3$, 4&$2^3$,
4&3&3&$D_4 \times \ZZ_2$&1&$\ZZ_2 \times \ZZ_4$&$\ZZ_2 \times \ZZ_4$\\
\hline 4&2, 4, 5&    3,
$6^2$&4&21&$\mathfrak{S}_5$&1&$\ZZ_3^2$&$\ZZ^2
\rtimes \ZZ_3$\\ 4&2,
$5^2$&$2^2$, $3^2$&4&11&$\mathfrak{A}_5$&1&$\ZZ_{15}$&$\ZZ_{15}$\\
4&2, 4, 6&$2^2$,
$4^2$&3&13&$\mathfrak{S}_4 \times \ZZ_2$&1&$\ZZ_2^2 \times
\ZZ_4$&$\ZZ^2 \rtimes \ZZ_4$\\
4&2,
4, 6&       $2^5$&3&13&$\mathfrak{S}_4 \times 
\ZZ_2$&1&$\ZZ_2^3$&$\ZZ^2 \rtimes \ZZ_2$\\
4&$2^3$, 4&    $2^3$, 4&5& 5&$\ZZ_2^4 \rtimes
\ZZ_2$&1&$\ZZ_4^2$&$G(32,2)$\\ 4&3,
$4^2$&       $2^5$&3& 7&$\mathfrak{S}_4$&1&$\ZZ_2^2 \times
\ZZ_4$&$\ZZ^2 \rtimes \ZZ_4$\\
4&3, $6^2$&$2^2$, $3^2$&4& 4&$\mathfrak{S}_3 \times \ZZ_3$&1&$\ZZ_3^2$&$\ZZ^2
\rtimes \ZZ_3$\\
4&$2^2$, $3^2$&$2^2$,
$3^2$&4& 4&$\ZZ_3^2
\rtimes \ZZ_2$&1&$\ZZ_3^3$&$\ZZ_3^3$\\
4&$2^3$, 4&       $2^5$&3&
5&$D_4 \times \ZZ_2$&1&$\ZZ_2^2 \times \ZZ_4$&$\ZZ^2 \hookrightarrow \pi_1
\twoheadrightarrow D_4$\\
4&$2^2$,
$4^2$&$2^2$, $4^2$&3& 3&$\ZZ_4
\times \ZZ_2$&1&$\ZZ_2^3 \times \ZZ_4$&$\ZZ^4 \hookrightarrow \pi_1
\twoheadrightarrow \ZZ_2^2$                   \\
4&$2^5$&
$2^5$&3&
3&$\ZZ_2^3$&1&$\ZZ_2^3 \times \ZZ_4$&$\ZZ^4 \hookrightarrow \pi_1
\twoheadrightarrow \ZZ_2^2$\\
\hline 6&2, $5^2$&$3^2$, 4&19&16&$\mathfrak{A}_6$           &
2&$\ZZ_{15}$&$\mathfrak{A}_4 \times \ZZ_5$\\
6&2, 4, 6 &2, 4, 10&11&19&$\mathfrak{S}_5 \times \ZZ_2$&1&$\ZZ_2 \times
\ZZ_4$&$\mathfrak{S}_3 \times D_{4,5,-1}$\\
6&2,
$7^2$&$3^2$, 4&19& 8&PSL(2,7)                   &
2&$\ZZ_{21}$&$\mathfrak{A}_4 \times \ZZ_7$\\
6&2, $5^2$&2, $3^3$&
4&16&$\mathfrak{A}_5$           &
1&$\ZZ_3 \times \ZZ_{15}$&$\ZZ^2 \rtimes \ZZ_{15}$\\
6&2, 4, 6 &$2^4$, 4&
3&19&$\mathfrak{S}_4 \times \ZZ_2$&  1&$\ZZ_2^3 \times \ZZ_4$&$\Pi_2 
\hookrightarrow \pi_1
\twoheadrightarrow \ZZ_2 \times \ZZ_4$ \\
6&$2^3$, 4&$2^4$, 4& 3&
7&$D_4 \times \ZZ_2$           &  1&$\ZZ_2^2 \times \ZZ_4^2$&$\ZZ^2
\times \Pi_2 \hookrightarrow \pi_1
\twoheadrightarrow \ZZ_2^2$\\
\hline
\end{tabular}
\end{table}

\medskip Let's spend here some  words in order to give
an idea of the methods used for the proof.

  First of all we
use the combinatorial restriction imposed by the further assumption
$p_g(S) = 0$, and the conditions on the singularities of $X$.

This allows, for each possible value of $K^2 : = K^2_S$, to restrict 
to a finite
list of possible signatures $T_1$, $T_2$ of the respective polygonal groups.

A finite but rather big list is provided by lemma \ref{boundrm}; moreover  a
MAGMA (\cite{magmaref}) script allows us to shorten the list to at most $24$
signatures for each value of $K^2$.

The order of $G$ is now determined by $T_1$, $T_2$ and by
  $K^2$: it follows that there are only finitely many groups to
  consider.

A second MAGMA script computes, for each
$K^2$, all possible triples $(T_1,T_2,G)$, where $G$ is a quotient of both
polygonal groups (of respective signatures $T_1,T_2$) and has the right
  order.
Note that our code skips a few pairs of signatures
giving rise to groups of large order, either not covered by the MAGMA
SmallGroup database, or causing extreme computational complexity.
These cases left out by our program are then excluded via a case by
  case argument.

% 1024, or bigger than 2000, which are not
%covered by the MAGMA SmallGroup database of finite groups.
%The code skips also the case $|G|=1152$, since there are more than
%$10^6$ groups of this
%order and this causes extreme computational complexity.  All the
%cases left out by our
%program are then excluded via case by case computer calculation.

For each of the triples $(T_1,T_2,G)$ which we have then found
there are the corresponding pairs of surjections each giving a family
  of surfaces of
the form $(C_1 \times C_2)/G$.  Some of these surfaces have
singularities which   violate the hypotheses of theorem \ref{classpgq=0}.

A third MAGMA script produces the final list of
surfaces, discarding the ones which are too singular.

A last script calculates, using prop. \ref{fundg}, the
fundamental groups.

As we already remarked, in the case $K^2_S = 2$  (of the so called 
{\em numerical Campedelli surfaces})
all fundamental groups that we get are finite; but for $K^2_S = 4, 6$ 
both cases
occur:  finite and  infinite fundamental groups.

The case of infinite
fundamental groups is the one where the structure theorem proven in 
the first part of
the paper turns out to be extremely helpful to give an explicit 
description of these groups
(since in general a presentation of a group does not say much about it).

\section{Notation}

In this section we collect the notation we use throughout the paper.

$G$: a finite group.

$C_i $: a smooth compact (connected) curve  of genus $g_i \geq 2$.

Given  natural numbers $g \geq 0; \ m_1, \ldots , m_r \geq 1$ the {\em
   orbifold surface group}  of {\em signature} $(g; m_1, \ldots , m_r)$
is defined as follows:
\begin{align*}
\BT(g;m_1,\ldots, m_r):=\langle &a_1,b_1,\ldots ,a_g,b_g, c_1,\ldots,
c_r|\\
&c_1^{m_1},\ldots ,c_r^{m_r},\prod_{i=1}^g [a_i,b_i]
\cdot c_1\cdot
\ldots \cdot c_r\rangle.
\end{align*}

For $g=0$ we get the  {\em polygonal group}
\begin{equation}\label{polygroup}
\BT(0;m_1,\ldots, m_r)=\langle c_1,\ldots, c_r\
     | \ c_1^{m_1},\ldots ,c_r^{m_r},c_1\cdot \ldots \cdot c_r\rangle .
\end{equation}

For $r=0$ we get the {\em surface group} of genus $g$
\begin{equation*}
\Pi_g:=\langle a_1,b_1,\ldots , a_g,b_g\
     | \prod_{i=1}^g [a_i,b_i] \rangle .
\end{equation*}

An {\em appropriate orbifold homomorphism} $\varphi \colon
\BT(g';m_1,\ldots, m_r) \rightarrow
G$ is a surjective homomorphism such that
$ \ga_i : = \fie (c_i)$ has order {\em exactly} $m_i$.

Given two surjective homomorphisms
$$
\varphi_1 \colon \BT_1 \rightarrow G, \
\varphi_2 \colon \BT_2  \rightarrow G.
$$

\noindent

we define the fibre product $\HH $ as
\begin{equation}\label{fibreprod} \HH : = {\HH}(G;\varphi_1,\varphi_2):= \{ \,
(x,y)\in \BT_1\times \BT_2\
|\
\varphi_1(x)=\varphi_2(y)\,\}.
\end{equation}

\section{Finite group actions on products of curves}

The following
facts will be frequently used without explicit mention in the
subsequent chapters.

The following is a reformulation of {\em Riemann's existence theorem}:

\begin{theo} A finite group $G$ acts as a group of automorphisms on a
compact Riemann
surface $C$ of genus $g$ if and only if there are natural numbers
$g', m_1, \ldots ,
m_r$, and an appropriate  orbifold homomorphism
$$\varphi \colon
\BT(g';m_1,\ldots, m_r)
\rightarrow G$$  such that the Riemann - Hurwitz relation holds:
$$ 2g - 2 = |G|\left(2g'-2 + \sum_{i=1}^r \left(1 -
\frac{1}{m_i}\right)\right).
$$

\end{theo}

If this is the case, then $g'$ is the genus of $C':=C/G$. The
$G$-cover $C \rightarrow
C'$ is branched in $r$ points $p_1, \ldots , p_r$ with branching
indices $m_1, \ldots ,
m_r$, respectively.

\medskip\noindent Moreover, if we denote by $\ga_i \in G$ the image
of $c_i$ under
$\varphi$, then
$$
\Sigma(\ga_1, \ldots , \ga_r) := \cup_{a \in G} \cup_{i=0}^{\infty}
\{a\ga_1^ia^{-1}, \ldots a\ga_r^ia^{-1} \},
$$ is the set of stabilizers for the action of $G$ on $C$.

\medskip Assume now that there are two homomorphisms
$$
\varphi_1 \colon \BT(g'_1;m_1,\ldots, m_r) \rightarrow G,
$$
$$
\varphi_2 \colon \BT(g'_2;n_1,\ldots, n_s) \rightarrow G,
$$  determined by two Galois covers $\lambda_i \colon C_i \rightarrow
C'_i$, $i = 1, 2$.

\noindent We will assume in the following that $g(C_1), \ g(C_2) \geq
2$, and we consider
the diagonal action of $G$ on $C_1 \times C_2$.

If $G$ acts freely on $C_1 \times C_2$, then $S:=(C_1 \times C_2) /
G$ is smooth and is
said to be {\em isogenous to a product}. These surfaces were
introduced and extensively
studied by the second author in \cite{FabIso}, where the following
crucial {\em weak
rigidity} of surfaces isogenous to a product was proved (see also
\cite{modreal}).
\begin{theo}
     Let $S=C_1 \times C_2 /G$ be a surface isogenous to a product,
$g(C_1)$,  $g(C_2) \geq 2$.
Then every surface with the same
\begin{itemize}
\item[-]  topological Euler characteristic and
\item[-]  fundamental group
\end{itemize}
     is diffeomorphic to $S$. The corresponding  moduli space
$\mathcal{M}_S^{top} = \mathcal{M}_S^{diff}$ of surfaces (orientedly)
homeomorphic (resp.
diffeomorphic) to $S$ is either irreducible and connected or consists
of two irreducible
connected components exchanged by complex conjugation.
\end{theo}

In particular, any flat deformation of a surface isogenous to a
product is again
isogenous to a product.
Observe that this property does not hold any longer if
the action is not free.

Moreover, the fundamental group of $S=C_1 \times C_2 /G$ sits inside
an exact sequence

$$ 1 \rightarrow \pi_1(C_1) \times \pi_1(C_2) \rightarrow \pi_1 (X)
\rightarrow G
\rightarrow 1.
$$ This extension is determined by the associated maps to the
Teichm\"uller modular
groups.

\begin{rem} In \cite{bacat} and \cite{bcg} a complete classification 
of surfaces $S$ isogenous to a product with $p_g(S) = q(S) = 0$ is 
given.
\end{rem}

In this paper we drop the condition that the action of $G$ on $C_1 \times C_2$
is free and we are mainly interested in the following two questions:
\begin{itemize}
     \item[-] what is the fundamental group of $X:=(C_1 \times C_2)/G$ ?
\item[-] is it still possible to classify these quotients under suitable
restrictions on the
invariants of a minimal resolution of the singularities of $X$?
\end{itemize}

\begin{rem}\label{vanKampen} If the diagonal action of $G$ on $C_1
\times C_2$ is not
free, then
$G$ has a finite set of fixed points. The quotient surface $X:= (C_1
\times C_2) / G$ has
a finite number of (finite) cyclic quotient singularities, which are rational
singularities.

Since, as we will shortly recall, the minimal resolution  $S
\rightarrow X$ of the
singularities of $X$ replaces each singular point by a tree of smooth
rational curves, we have, by van Kampen's theorem, that
$\pi_1(X) = \pi_1(S)$.
\end{rem}

     Note that by a result of A. Fujiki (cf. \cite{fujiki}), the exceptional
divisors of a resolution of the singularities $\tilde{X}$ of $(C_1 \times
\ldots \times C_n)
/G$ consists, in the case where the singularities are isolated
points, of a union of
irreducible rational varieties intersecting with simple normal
crossings. Therefore
$$\pi_1 ((C_1 \times \ldots \times C_n) /G) \cong \pi_1(\tilde{X}).
$$
\begin{rem}
1) Assume that $x \in X$ is a singular point. Then it is a cyclic
quotient singularity
      of type $\frac{1}{n}(1,a)$ with $g.c.d(a,n) = 1$,
       i.e., $X$ is locally around $x$ the quotient of $\mathbb{C}^2$ by
       the action of a diagonal linear automorphism with eigenvalues
       $\exp(\frac{2\pi i}{n})$, $\exp(\frac{2\pi i a}{n})$.
This follows since the  tangent representation is faithful on both factors.

The particular case where $a=-1$, i.e., the stabilizer has a tangent
representation
with determinant $=1$, is precisely the case where the singularity is
a RDP (Rational
Double Point) of type $A_{n-1}$.

2) We denote by $K_X$ the canonical (Weil) divisor on the normal surface
corresponding to $i_* ( \Omega^2_{X^0})$, $ i\colon X^0 \ra X$ being the
inclusion of the
smooth locus of $X$. According to Mumford we have an intersection
product with values
in $\QQ$ for Weil divisors on a normal surface, and in particular we consider
the selfintersection  of the canonical divisor,
\begin{equation}\label{K2}
K_X^2 =
\frac{8 (g(C_1) - 1) (g(C_2) - 1)}{|G|}
\in
       \mathbb{Q},
\end{equation}
  which is not necessarily an integer.

$K_X^2$ is however an integer (equal indeed to $K_S^2$) if $X$
has only RDP's  as singularities (cf. \cite{reid}).
\end{rem}

The resolution of a cyclic quotient singularity of type $\frac{1}{n}(1,a)$ with
$g.c.d(a,n) = 1$ is well known. These singularities are resolved by
the so-called {\em
Hirzebruch-Jung strings}. More precisely, let $\pi \colon S \rightarrow X$
be a minimal
resolution of the singularities  and let $E = \bigcup_{i=1}^m E_i =
\pi^{-1}(x)$. Then $E_i$
is a smooth rational curve with $E_i^2 = -b_i$ and $E_i\cdot E_{i+1}
= 1$ for $i \in\{1,
\ldots , m-1\}$ and zero otherwise. The $b_i$'s are given by the formula
$$
\frac{n}{a} = b_1 - \frac{1}{b_2 - \frac{1}{b_3 - \ldots}}.
$$

Moreover, we have (around $x$)
$$ K_S = \pi^* K_X + \sum_{i=1}^m a_i E_i,
$$  where the rational numbers $a_i$ are determined by the conditions
     $$(K_S + E_j)E_j =
-2, \ \ \
(K_S -
\sum_{i=1}^m a_iE_i)E_j = 0, \ \
\forall j= 1, \dots ,m.$$

The {\em index} $r$ of the singularity $x$ is now given by

$$r = \min
\{\lambda \in \mathbb{N} | \lambda a_i \in \mathbb{Z}, \  \forall i=
1, \dots ,m\}.$$

\medskip\noindent Observe that the above formulae allow to calculate
the self intersection
number of the canonical divisor $K_S$ of a minimal resolution of the
singularities of
$X$.

\noindent The next result gives instead  a formula for the topological Euler
characteristic $e(S)$ of a minimal resolution $S$ of singularities of $X$.
\begin{prop}\label{euler} Assume that $\{p_1,
\ldots p_k\}$ are the cyclic quotient singularities of $X = C_1 \times C_2/G$,
of respective types $\frac{1}{n_i}(1,a_i)$. Let $p \colon C_1 \times
C_2 \ra X$ be the quotient morphism and let $S$ be a minimal 
resolution of the singularities of $X$.
Denote by $l_i$
the length of the resolution tree of the singularity $p_i$. Then
$$ e(S) = \frac{K_X^2}{2} - \frac{|p^{-1}(\{p_1, \ldots p_k\})|}{|G|}
+ \sum_{i=1}^k (l_i
+ 1).
$$
\end{prop}

\Proof Let $X^* := X \setminus \{p_1, \ldots , p_k \}$. Then,
using the additivity of the Euler number for a stratification given
by orientable
manifolds (see e.g. \cite{FabIso}), we  obtain
$$e(S) = e(X^*) +
\sum_{i=1}^k (2l_i - (l_i -1)) = e(X^*) + \sum_{i=1}^k (l_i + 1).$$
Let $Z^* := (C_1
\times C_2) \setminus p^{-1}(\{p_1, \ldots p_k\})$: then $p|Z^* \colon
Z^* \rightarrow X^*$ is
an \'etale Galois covering with group $G$. Therefore:
\begin{multline*} e(X^*) = \frac{e(Z^*)}{|G|} =
\frac{e(C_1 \times C_2) -|p^{-1}(\{p_1, \ldots p_k\})|}{|G|} = \\
\frac{4(g(C_1)-1)(g(C_2) -1) -|p^{-1}(\{p_1,
\ldots p_k\})|}{|G|} = \frac{K_X^2}{2}  - \frac{|p^{-1}(\{p_1, \ldots
p_k\})|}{|G|}.
\end{multline*}

\qed

An immediate consequence of the previous proposition is the following:
\begin{cor}\label{eulernodes}
     Assume that the singular points $p_1,
\ldots p_k$
of $X = (C_1 \times C_2)/G$ are ordinary double points
(i.e., $A_1$ singularities).
     Let $S$ be a minimal resolution of the singularities of $X$. Then
$$ e(S) = \frac{K_X^2}{2} + \frac{3}{2} k.
$$
\end{cor}

\Proof
Here, for each $i$, $l_i = 1$ and $|p^{-1}(p_i)| = \frac{|G|}{2}$.

\qed

%%%%%%%%%%%%%%%%%%%%%%%%%%%%%%%%%%%%%%%%%%%%%%%%%%%%%%%%%%%
\section{From the geometric  to the algebraic set up for calculating 
the fundamental group.}\label{fundgr}
%%%%%%%%%%%%%%%%%%%%%%%%%%%%%%%%%%%%%%%%%%%%%%%%%%%%%%%%%%%

Assume that $X = ( C_1 \times C_2 )/G$, where $G$ is a finite group 
of automorphisms of each factor $C_i$ and acts diagonally on $C_1 
\times C_2$ (for short: the action of $G$ on $C_1 \times C_2$ is {\em 
unmixed}).

\begin{defin} Let $X$ be as above and consider the minimal resolution
$S$ of the singularities of $X$.
The holomorphic map $ f_1\colon S \ra C_1' : = C_1 / G$ is called  a {\em
standard isotrivial fibration} if it is a relatively minimal fibration.

In the general case one lets $f' \colon S'\ra C_1'$ be the relatively 
minimal model
of $f_1$, and says that $f'$ is an isotrivial fibration.
\end{defin}

As already observed in remark \ref{vanKampen},  we have $\pi_1(X) =
\pi_1(S)$.

\medskip
The aim is now to determine the fundamental group of $X$ in terms
of the following algebraic data:

i)  the group $G$ together with

ii) the two surjective homomorphisms
$$
\varphi_1\colon \BT_1 : = \BT(g'_1;m_1,\ldots, m_r) \rightarrow G, \ \
\varphi_2\colon  \BT_2 : = \BT(g'_2;n_1,\ldots, n_s) \rightarrow G.
$$

\noindent

\begin{rem} The surjectivity of the homomorphisms $\varphi_1$ and
$\varphi_2$ implies
that for each $h_1\in \BT_1$ there exists an element $h_2 \in\BT_2$ such that
$(h_1,h_2) \in {\HH}$, where ${\HH} : = {\HH}(G;\varphi_1,\varphi_2)$ 
is the fibre product as defined in (\ref{fibreprod}).
\end{rem}

\noindent The exact sequence
\begin{equation}\label{ex1} 1 \rightarrow \Pi_{g_1} \times \Pi_{g_2}
\rightarrow \BT_1
\times
\BT_2
\rightarrow G \times G \rightarrow 1,
\end{equation} where $\Pi_{g_i}   := \pi_1(C_i)$, induces an exact
sequence
\begin{equation}\label{ex2} 1 \rightarrow  \Pi_{g_1} \times \Pi_{g_2}
\rightarrow {\HH}
\rightarrow G \cong \Delta_G \rightarrow 1.
\end{equation} Here $\Delta_G \subset G \times G$ denotes the diagonal.

\begin{definition}
     Let $H$ be a group. Then its {\em torsion subgroup} $\Tors(H)$
is the (normal)
subgroup  generated by all elements of finite order in $H$.
\end{definition}

We have the following

\begin{prop}\label{fundg} Let $G$ be a finite group and let
$$\varphi_1 \colon \BT_1 : = \BT(g'_1;m_1,\ldots, m_r) \rightarrow G,
\ \
\varphi_2 \colon \BT_2
:= \BT(g'_2;n_1,\ldots, n_s) \rightarrow G
$$ be two surjective homomorphisms. Consider the induced
action of $G$ on
$C_1 \times C_2$.

Then $\pi_1((C_1 \times C_2)/G) \cong {\HH}/
{\Tors}(\HH)$.
\end{prop}

\Proof
It follows from the main theorem of \cite{armstrong1}, \cite{armstrong2}
since the elements of finite order are precisely those elements of
$\HH$ which have fixed points.

\qed

\section{The structure theorem for fundamental groups of quotients of
products of curves}

The aim of this section is to prove the following result

\begin{theo}\label{strfund}
Let $\BT_1, \ldots , \BT_n$ be orbifold surface groups and assume
that there are
surjective homomorphisms $\varphi_i \colon \BT_i \rightarrow G$ to a
finite group $G$. Let
\begin{align*}
\HH &: = {\HH}(G;\varphi_1, \ldots , \varphi_n)\\
&:= \{ \, (x_1, \ldots
,x_n)\in \BT_1
\times
\ldots \times \BT_n\  |\
\varphi_1(x_1)= \ldots =\varphi_n(x_n)\,\}
\end{align*}
be the fibre product of
$\varphi_1 , \ldots ,
\varphi_n$.

     Then ${\HH} / {\Tors}({\HH})$ has a normal subgroup $\mathcal{N}$
of finite index
isomorphic to the product of surface groups $\Pi_{h_1} \times \ldots
\times \Pi_{h_n}$.
\end{theo}

In particular, we have an exact sequence
$$ 1  \rightarrow \Pi_{h_1} \times \ldots \times \Pi_{h_n}
\rightarrow {\HH} / {\Tors}({\HH}) \ra G'
\ra 1.
$$

where $G'$ is a finite group.

\begin{rem}
     If $G$ acts freely on $C_1 \times \ldots \times C_n$, or in other words, if
${\Tors}({\HH}) = \{ 1\}$, this is just the exact
sequence (\ref{ex2}) (here $G = G'$). Morally speaking, our theorem
means that even
admitting fixed points of the action of $G$ the structure of the
fundamental group of the
quotient of a product of curves by $G$ (or equivalently of a minimal
resolution of
singularities) is not different from the \'etale case.
%Of course, in the non etale case the normal subgroup as well as the
%group $G'$ are by no means unique.
\end{rem}

In order to keep the notation down to a reasonable level we shall
give the proof of theorem (\ref{strfund}) only for the case $n =
2$. The proof for the general case is exactly the same.

\medskip

Let now $g', m_1, \ldots , m_r$ be natural numbers and consider an
orbifold surface
group
$$
\BT:=\BT(g';m_1, \ldots , m_r).
$$

In the sequel we will frequently use the following well known
properties of orbifold
surface groups:

\begin{prop}
i) Every element of finite order in $\BT$ is equal to  a
conjugate of a suitable power of one of the generators $c_i$ of finite order.

ii) Every orbifold surface group contains a surface group of  finite index.
\end{prop}

\Proof Cf. \cite{beardon}.

\qed

\begin{definition}

Let $R \lhd \BT$ be a normal subgroup of finite index and let $L$ be
an arbitrary subset of  $\BT$. We define
\begin{equation} {\rm N}(R,L):=  \langle  \langle\  \{ \, h k h^{-1}
k^{-1} \ | \  h\in
L,\  k\in R\, \}\ \rangle  \rangle_{\BT}
\end{equation}  to be the normal subgroup in $\BT$ generated by the
set
$
\{ \, h k h^{-1} k^{-1}  \ |\ h\in L,\  k\in R\,\},
$  and denote the corresponding quotient by
$$
\hat \BT := \hat \BT(R,L):=\BT / {\rm N}(R,L).
$$
\end{definition}

The centralizer of the image of $L$ in $\hat \BT(R,L)$ is a finite index
subgroup of $\hat \BT$.

\begin{prop}\label{Pro6} Let $R \lhd \BT$ be a normal subgroup of
finite index and  let $L$ be a finite subset of $\BT$ consisting of
elements  of finite order.    Then the normal subgroup $ \langle
\langle L \rangle
\rangle _{\hat \BT}$, generated in
$\hat \BT$ by the image of the set $L$, is finite.
     \end{prop}

The proof of proposition \ref{Pro6} follows easily from the following lemma.
Even though this lemma is a consequence of a deep theorem of Schur
(see \cite{Huppert}),
we prefer to give a short and self contained proof.

\begin{lemma}\label{Le5} Let $H$ be a group such that
\begin{itemize}
\item[i)] $H$ is generated by finitely many elements of finite order,
\item[ii)] the center ${\rm Z}(H)$ has finite index in $H$.
\end{itemize} Then $H$ is finite.
\end{lemma}

\Proof By ii) it suffices to show that ${\rm Z}(H)$ is finite.
Observe first that since ${\rm Z}(H)$ has finite index, it is
finitely generated.

Assume to the contrary that ${\rm Z}(H)$ is infinite. Writing
     the multiplication of ${\rm Z}(H)$ additively, we denote by $m{\rm
Z}(H)$ the subgroup
of ${\rm Z}(H)$ consisting of $m$-th powers. Since ${\rm Z}(H)$ is an
infinite,  finitely
generated abelian group we infer that ${\rm Z}(H)/ m{\rm Z}(H)$ is
non trivial for every
$m\in \ZZ \setminus \{1,-1\}$. Note also that $m{\rm Z}(H)$ is normal in
$H$.

We choose $m$ now coprime to the order of $H / {\rm Z}(H)$ and also
coprime  to the
orders of the generators of $H$. Consider the exact sequence
\begin{equation}
1 \to {\rm Z}(H) / m{\rm Z}(H)\to H / m {\rm Z}(H)\to
H / {\rm Z}(H) \to 1.
\end{equation} By our choice of $m$ this sequence splits (cf.
\cite{Huppert}), whence
$ H / m {\rm Z}(H)$ is a semi-direct product of
${\rm Z}(H) / m{\rm Z}(H)$ and $H / {\rm Z}(H)$.

Since
${\rm Z}(H) / m{\rm Z}(H)$ is central in $ H / m {\rm Z}(H)$ we infer
that $ H / m {\rm
Z}(H)$ is isomorphic to the direct product of
${\rm Z}(H) / m{\rm Z}(H)$ and $H / {\rm Z}(H)$. Therefore there is a
surjective
homomorphism
$H / m {\rm Z}(H)\to {\rm Z}(H) / m {\rm Z}(H)$, a contradiction to
the fact that $m$ is coprime to the order of each generator of $H$.

\qed

\medskip

\Proof(of prop. (\ref{Pro6})).
Denote by $\hat R$ the image of $R$ in $\hat \BT$, and recall that
$c_1,\ldots,c_r$ are the generators of finite order of $\BT$.

$L$ consists of finitely many elements, each a conjugate of a $c_i^{l_i}$.
Their image in $\hat \BT$ is centralised by $\hat R$. Since $\hat R$
has finite index in $\hat \BT$, finitely many conjugates of the
elements of $L$ suffice to generate $\langle \langle L \rangle
\rangle_{\hat \BT}$.

Hence  the group $\langle \langle L \rangle\rangle_{\hat \BT}$ is 
generated by finitely many elements of finite order.

Now, $\hat R \cap \langle \langle L \rangle\rangle_{\hat \BT}$ has 
finite index in $\langle \langle L \rangle\rangle_{\hat \BT}$. By 
construction
$\hat R \cap \langle \langle L \rangle\rangle_{\hat \BT}$ is a 
central subgroup of $\langle \langle L \rangle\rangle_{\hat \BT}$. 
Therefore  the second
condition of lemma
\ref{Le5} is fulfilled, and we conclude that $\langle \langle L 
\rangle\rangle_{\hat \BT}$ is  finite.

\qed

\begin{lemma}\label{Pro7} Let $\BT(g';m_1,\ldots, m_r)$ be an
orbifold surface group
and  let $L$ be a subset of  $\BT(g';m_1,\ldots, m_r)$ consisting of
elements  of finite
order.

Then there are $k_1,\ldots , k_r \in \mathbb{N}$, such that $ k_i |
m_i \ \forall
i=1,\dots r$ such that
\begin{equation}
\BT(g';m_1,\ldots, m_r) / \langle \langle L
\rangle \rangle _{\BT(g';m_1,\ldots, m_r)}\cong \BT(g';k_1,\ldots, k_r).
\end{equation}  In particular, the quotient group $\BT(g';m_1,\ldots,
m_r) / \langle
\langle L
\rangle \rangle _{\BT(g';m_1,\ldots, m_r)}$ is again an orbifold surface group.
\end{lemma}

\Proof
The normal subgroup $\langle \langle L
\rangle \rangle _{\BT(g';m_1,\ldots, m_r)}$ is  normally generated by
a set of the form
$\{c_1^{k_1}, \ldots , c_r^{k_r} \}$ with $ k_1, \ldots , k_r \geq 1$,
$ k_i \leq m_i$, $ k_i | m_i$.

\qed

\begin{rem} Let $R \lhd \BT$ be a normal subgroup of finite index and 
let $L$ be
an arbitrary subset of  $\BT$ consisting of elements of finite order.

Note that
$$ E(R,L):=\langle \langle L
\rangle \rangle_{\hat \BT(R,L)}
$$
     is a finite normal subgroup of $\hat \BT(R,L)$ and
$$ N(R,L)\lhd \langle \langle L
\rangle \rangle_{\BT}.
$$

Hence we have
$$
\hat \BT(R,L) / \langle \langle L
\rangle \rangle_{\hat\BT(R,L)}\cong \BT / \langle
\langle L
\rangle \rangle_{\BT}.
$$
\end {rem}

We want to apply the above general considerations to our situation.
For this purpose we have to fix some more notation.

\medskip\noindent We fix two orbifold surface groups
\begin{align*}
\BT_1&:=\BT(g'_1;m_1,\ldots, m_r)=\\
     & = \langle a_1,b_1,\ldots , a_{g_1'},b_{g_1'}, c_1,\ldots, c_r |
c_1^{m_1},\ldots,c_r^{m_r},\prod_{i=1}^{g_1'} [a_i,b_i]\cdot c_1\cdot 
\ldots \cdot
c_r \rangle,
\end{align*}
and
\begin{align*}
\BT_2&:=\BT(g'_2;n_1,\ldots, n_s)=\\
     & = \langle a'_1,b'_1,\ldots , a'_{g_2'},b'_{g_2'}, d_1,\ldots, d_s |
d_1^{n_1},\ldots,d_s^{n_s},\prod_{i=1}^{g_2'} [a'_i,b'_i]\cdot 
d_1\cdot \ldots \cdot
d_s \rangle,
\end{align*}
together with two surjective homomorphisms $\varphi_1 \colon \BT_1
\rightarrow G$,
$\varphi_2 \colon \BT_2 \rightarrow G$ to a (nontrivial) finite group $G$.

\noindent Denote by
\begin{equation*} R_1:={\rm Ker}(\varphi_1) \lhd \BT_1 ,\qquad R_2:={\rm
Ker}(\varphi_2) \lhd \BT_2
\end{equation*} the respective kernels. If $\varphi_1$ and $\varphi_2$
are appropriate
orbifold homomorphisms, then
$R_1$ and
$R_2$ are both isomorphic to surface groups (else, they are just
orbifold surface groups).

      Define further
$$C:=\{\, c_1,\ldots,c_r\,\}\subset \BT_1,\qquad  D:=\{\,
d_1,\ldots,d_s\,\}\subset
\BT_2.$$

\begin{lemma}\label{Le6} Let $G$ be a finite group and let $\varphi_1
\colon \BT_1 \to G$,
$\varphi_2 \colon \BT_2 \to G$ be two surjective group homomorphisms.

\noindent 1) Then there is a finite set ${\sN}_1 \subset \BT_1 \times \BT_2$ of
elements of the form
$$(c^{l}, z d^{n} z^{-1}) \in \BT_1 \times \BT_2 ~ ,  \qquad c\in
C,\, d\in D,\, l,\,
n\in\NN,\,  z\in\BT_2,
$$ which have the property that
\begin{itemize}
\item[a)] ${\sN}_1\subset {\HH}:={\HH}(G;\varphi_1,\varphi_2)$,
\item[b)] the normal closure of ${\sN}_1$ in ${\HH}$  is
equal to ${\Tors}({\HH})$.
\end{itemize}
\noindent 2) Similarly  there is ${\sN}_2  \subset \BT_1 \times \BT_2$
as in 1) exchanging the roles of  $\BT_1 $ and $ \BT_2 $.
Then the two sets  ${\sN}_1, \  {\sN}_2$ can be moreover
chosen in such a
way that   the following further condition holds:

if $(c^{l}, z d^{n} z^{-1})$ is an element of ${\sN}_1$, there is a
$z_1\in \BT_1$ such that $(z_1 c^{l}z_1^{-1}, d^{n})$ is in ${\sN}_2$,  and
conversely, if  $(z_1 c^{l}z_1^{-1}, d^{n})$ is in ${\sN}_2$, then
there is $z_2\in
\BT_2$, such that $(c^{l}, z_2 d^{n} z_2^{-1})$  is an element of ${\sN}_1$.
\end{lemma}

\Proof Every element of finite order in
${\HH}$ is of the form
$$ (z_1c^l z_1^{-1},z_2d^nz_2^{-1}),
$$  where $c \in C$, $d \in D$, $l,n \in \NN$ and $z_i \in \BT_i$.
Since there is an
element in
${\HH}$ of the form $(z_1,f)$ we can say that  every element of
finite order in
${\HH}$ is conjugate in ${\HH}$ to an
element of the form $(c_i^l,zd_j^nz^{-1})$.

\qed

\begin{lemma}\label{Le7} Under the same assumptions as in lemma
(\ref{Le6}), the following
holds:
\begin{itemize}
     \item[i)] if $(c^{l}, z d^{n} z^{-1}) \in {\sN}_1$ (for $c\in C$,
$d\in D$, $l,\,
n\in\NN$, $z\in\BT_2$), then
$$ (c^{l} k c^{-l} k^{-1},1) \in {\Tors}({\HH})
$$ for all $k\in R_1$;
\item[ii)] if $(z c^{l} z^{-1}, d^{n}) \in {\sN}_2$ (for $c\in C$,
$d\in D$, $l,\,
n\in\NN$, $z\in\BT_1$), then
$$ (1,d^{n} k d^{-n} k^{-1}) \in {\Tors}({\HH})
$$ for all $k\in R_2$.
\end{itemize}
\end{lemma}

\Proof We have $(k,1)\in {\HH}$ for every
$k\in R_1$. From
$$ (c^{l} k c^{-l} k^{-1},1)=(c^{l}, z d^{n} z^{-1})\cdot (
(k,1)\cdot (c^{l}, z d^{n}
z^{-1})^{-1}\cdot (k,1)^{-1} )
$$ the first claim follows. The second claim follows by symmetry.

\qed

\begin{definition} We define $L_i \subset \BT_i$ as follows:
$L_1$ is the set of first components of elements of $\sN_1$,
$L_2$ is the set of second components of elements of $\sN_2$.

\end{definition}

Now the two homomorhisms $\varphi_1 \colon \BT_1 \to G$,
$\varphi_2 \colon \BT_2 \to G$ induce surjective homomorphisms
$$\hat{\varphi_1} \colon \hat\BT_1:=\hat\BT(R_1,L_1) \to G, \ \
\hat\varphi_2 \colon \hat\BT_2:=\hat\BT(R_2,L_2)\to G.
$$

Let's take then the fibre product
\begin{equation*}
\hat{\HH}:=\hat{\HH}(G;\hat\varphi_1,\hat\varphi_2):= \{ \, (x,y)\in 
\hat\BT_1\times
\hat\BT_2\  |\
\hat\varphi_1(x)=\hat\varphi_2(y)\,\}.
\end{equation*}
We shall define now a homomorphism
\begin{equation*}
\Phi\colon \hat{\HH} \to {\HH}/{\Tors}({\HH})
\end{equation*} as follows: for $(x_0,y_0)\in\hat{\HH}$ choose
$(x,y)\in \BT_1\times\BT_2$  such that $x$ maps to $x_0$ under the
quotient homomorphism $\BT_1\to\hat\BT_1$ and
$y$ maps to $y_0$ under the  quotient homomorphism $\BT_2\to\hat\BT_2$.

We have
$\varphi_1(x)=\hat\varphi_1(x_0)$, $\varphi_2(y)=\hat\varphi_2(y_0)$,
hence $(x,y)\in {\HH}$. We set
\begin{equation*}
\Phi((x_0,y_0)):=(x,y)\ {\rm mod}\ {\Tors}({\HH}).
\end{equation*}
Lemma \ref{Le7} shows that $\Phi$ is well defined.
Obviously, $\Phi$ is  a
homomorphism.

We have

\begin{lemma}\label{Le8}
$\Phi\colon \hat{\HH}\to {\HH} /
{\Tors}({\HH})$ is surjective. Its kernel
is contained in
$E(R_1,L_1)\times E(R_2,L_2)$.

In particular, we have an exact sequence
\begin{equation}\label{ee} 1 \to E\to
\hat{\HH} \to
{\HH} / {\Tors}({\HH})
\to 1,
\end{equation}

where $E$ is a finite group.
\end{lemma}
\Proof
$\Phi$ is clearly surjective. Let $(x_0,y_0) \in \ker(\Phi)$. Then
$(x,y) \in  {\Tors}(\HH)$, hence $x_0$ is an element of $E(R_1,L_1)$ 
by the definition of
$E(R_1,L_1)$.
Analogously,
$y_0
\in E(R_2,L_2)$.

\qed

     We are now ready to prove theorem (\ref{strfund}).

The quotient homomorphisms $\hat \BT_1\to \hat \BT_1/\langle \langle L_1\rangle
\rangle_{\hat
\BT_1}
\cong
\BT_1/\langle \langle L_1\rangle \rangle_{\BT_1}$, $\hat \BT_2\to
\BT_2/\langle \langle
L_2
\rangle
\rangle_{\BT_2}$ both have finite kernel. They may be put together to give a
a homomorphism
\begin{equation*}
\Psi\colon \hat \BT_1\times \hat \BT_2\to \BT_1/\langle \langle L_1\rangle
\rangle_{\BT_1} \times
\BT_2/\langle \langle L_2\rangle \rangle_{\BT_2},
\end{equation*} again with finite kernel. Notice that
\begin{equation}\label{ent} E\le {\rm ker}(\Psi)
\end{equation} where $E$ is the kernel of $\Phi$, as  in (\ref{ee}).

We need the following:

\begin{rem} The subgroup
$${\mathbf H}:=\Psi(\hat{\HH})\le
\BT_1/\langle \langle L_1\rangle\rangle_{\BT_1} \times \BT_2/\langle \langle
L_2\rangle \rangle_{\BT_2}$$
is of finite index.
\end{rem}

Hence, we obtain an exact sequence
\begin{equation} 1 \to E_1\to \hat{\HH} \to
\mathbf H
\to 1,
\end{equation}
where $E_1$ is finite.

By (\ref{ent}) $E \leq E_1$, and by (\ref{ee}) we obtain a commutative
diagram with exact rows and columns

$$
\xymatrix{
&1\ar[d]&1\ar[d]&&\\
&E\ar[d]\ar@{=}[r]&E\ar[d]&&\\
1\ar[r]&E_1\ar[d]\ar[r]&\hat{\HH}\ar[d]\ar[r]&\mathbf H\ar@{=}[d]\ar[r]&1\\
1\ar[r]&E_2\ar[d]\ar[r]&\HH/\Tors(\HH)\ar[d]\ar[r]&\mathbf H\ar[r]&1\\
&1&1&&.
}
$$

Therefore we have
\begin{equation}\label{exa} 1 \to E_2 \to   {\HH} / {\Tors}({\HH})\to \mathbf H
\to  1
\end{equation}

where  $E_2$ is finite.

Moreover, $\mathbf H$, as a subgroup of finite index in the direct product
$\BT_1/\langle \langle L_1\rangle \rangle_{\BT_1} \times \BT_2/\langle
\langle L_2\rangle\rangle_{\BT_2}$ of orbifold surface groups, is 
residually finite.

\noindent But this property is not sufficient for our purposes. We
need the following
notion of {\em cohomological goodness} which was introduced by
J.P. Serre in \cite{Serre} and is very important for the comparison of a group
with its profinite completion.

\begin{defin}
     Let $H$ be a group and let $\hat{H}$ be its profinite completion.
The group $H$ is
called {\em good}, if, for each $k \in \mathbb{N}, k \geq 0$, and for
each finite $H$-module
$M$ the natural homomorphism of cohomology groups (induced by the
homomorphism $H
\rightarrow
\hat{H}$)
$$ H^k(\hat{H}, M) \rightarrow H^k(H,M)
$$ is an isomorphism.
\end{defin}

In \cite{GZ} it is shown that a direct product of good groups is good,
that a subgroup of finite index of a good group is again good, whence it
follows that $\mathbf H$ is
also good (since orbifold surface groups are good). This allows us to apply the
following result.

\begin{prop} (\cite[Prop. 6.1]{GZ}).
     Let $F$ be a residually finite, good group, and let $\phi \colon H
\rightarrow F$ be a
surjective homomorphism with finite kernel $K$. Then $H$ is residually finite.
\end{prop}

This implies that ${\HH} / {\Tors}({\HH})$ is
residually finite. Therefore (and because $E_2$ is finite) there is a
normal subgroup of
finite index
$$\Gamma\lhd {\HH} / {\Tors}({\HH})$$
such that $\Gamma\cap E_2=\{1\}$.

\noindent Let $\Psi_1\colon {\HH} / {\Tors}({\HH})\to \mathbf{H}$ be 
the surjective
homomorphism in the
exact sequence (\ref{exa}). Then $\Psi_1 | \Gamma$ is injective and clearly
$\Psi_1(\Gamma)$ has finite index in $\mathbf{H}$, which in turn  is
a subgroup of finite
index in a direct product of orbifold surface  groups.

\noindent From the general properties of orbifold surface groups, we
find a normal
subgroup
$\mathbf{H}_1\lhd \mathbf{H}$ of finite index, such that 
$\mathbf{H}_1$ is contained in
$\Psi_1(\Gamma)$ and isomorphic to the direct  product of two surface groups.

\noindent We set $\Gamma_1 := \Psi_1^{-1}(\mathbf{H}_1) \cap \Gamma$.
Then $\Gamma_1$
satisfies:

\begin{itemize}
\item[i)] $\Gamma_1$ is a normal subgroup of finite index in
${\HH} /  {\Tors}({\HH})$,
\item[ii)]  $\Gamma_1$ is isomorphic to the direct  product of two
surface groups.
\end{itemize} This proves theorem (\ref{strfund}).

\qed

We want to explain briefly the geometry underlying the above structure theorem
for the fundamental group, and its significance for the investigation 
of the moduli
space of the product-quotient surfaces.

\smallskip

We have the quotient map $ C_1 \times C_2   \ra X = (C_1 \times C_2 ) / G$,
and the minimal model $S$ of $X$ is such that $\pi_1(X) \cong \pi_1(S)$.

Let $\hat{X}$ be the unramified covering of $X$ corresponding to the
normal subgroup
$\sN$ of $\pi_1(X)$ of finite index such that
$$ \pi_1( \hat{X} ) = \sN \cong  \Pi_{h_1} \times \Pi_{h_2}.$$

The fibre product  $\hat{X} \times_X (C_1 \times C_2 )$ is an unramified Galois
covering of $(C_1 \times C_2 ) $ with group $ G' := \pi_1(X) / \sN $, 
which is connected if
$\pi_1 (C_1 \times C_2 ) \cong  \Pi_{g_1} \times \Pi_{g_2}$
surjects onto $\pi_1(X)$. Let $Y$ be a connected component of 
$\hat{X} \times_X (C_1 \times C_2 )$:
it is a Galois covering of  $(C_1 \times C_2 ) $ with Galois group 
$G''$, the subgroup
of $G'$ stabilizing $Y$.

In turn the surjection  $ \psi \colon \Pi_{g_1} \times \Pi_{g_2} \ra 
G''$ determines
on the one hand two normal
subgroups $ G''_i : = \psi (\Pi_{g_i} )$ of $G''$ and two Galois 
unramified coverings
$$ C''_i \ra C_i = (C''_i  )/ G''_i .$$

On the other hand it shows that each element of $G''_1$ commutes with
every element of $G''_2$, and  $G''_1$, $G''_2$ generate $G''$.

Hence, setting $ G''' : = G''_1 \cap G''_2 $, we have a central extension:
$$ 1 \ra G''' \ra G'' \ra (G''_1 / G''') \times (G''_2 / G''')  \ra 1  .$$

We analyse  now the surjective morphism
$(C''_1 \times C''_2)  \ra \hat{X} $,  in order to derive some
interesting consequences.

\medskip

\underline{\bf Case ++:} assume here that $h_1, h_2 \geq 2$.

\medskip

Then, by the extension of the Siu-Beauville theorem given in \cite{fibred},
each  surjection $ \sN \ra  \Pi_{h_j}$ is given by
  a surjective  holomorphic map with connected fibres $ \varphi_j 
\colon \hat{X} \ra C'''_j  $,
  thus we get a surjective  holomorphic map
$ \varphi \colon \hat{X} \ra (C'''_1 \times C'''_2) $,
where $C'''_i$ has genus $h_i$, and
$ \varphi _*$ induces the above isomorphism $ \sN \cong  \Pi_{h_1} 
\times \Pi_{h_2}$.

The first geometrical consequence is that, since the holomorphic map
$f \colon (C''_1 \times C''_2) \ra (C'''_1 \times C'''_2)$ is by the 
rigidity lemma (see \cite{FabIso})
  necessarily of product type (w.l.o.g., $f(x_1,x_2) = (f_1 (x_1), f_2(x_2))$),
the foliations on $\hat{X}$ induced by the fibrations of $X$ onto 
$C_1 / G$ and  $C_2 / G$
are the same as the ones induced by the maps of $\hat{X}$ onto $C'''_1$
and $C'''_2$.

Moreover, $G'$ acts on $\hat{X}$, hence also on its Albanese variety, which
is the product $\Alb(C'''_1) \times \Alb(C'''_2) $. Thus $G'$ acts 
also on each $C'''_i$,
in particular  $C'''_i / G'$ maps onto $C_i / G$, and in view of the 
connectedness of the general fibre
we get that  $C'''_i / G' \cong C_i / G$.

Since the unramified $G'$-cover $\hat{X}$ and $ \varphi \colon 
\hat{X} \ra (C'''_1 \times C'''_2) $
are dictated by the fundamental groups, a natural way to study the 
deformations of $X$ is to study
the deformations of the morphism $ \varphi $ which are $G'$-equivariant.

\underline{\bf  Case +0:} assume here that $h_1  \geq 2$, $h_2 = 1$.

Then we have $\varphi_1 \colon \hat{X} \ra C'''_1  $ such that $ 
{\varphi_1} _*$ induces the above
epimorphism $ \sN \cong  \Pi_{h_1} $.

The morphism $\Phi_1 \colon \hat{X} \ra \Alb(C'''_1) : = J$  factors 
, by the universal property of the
Albanese map, through the Albanese map $\alpha \colon \hat{X} \ra A : 
= \Alb (\hat{X})$, hence
$A$ is isogenous to a product of $J$ with an elliptic curve.

Since however $ H_1 (\hat{X}, \ZZ) = H_1 (C'''_1, \ZZ) \times \ZZ^2$, 
we conclude
that  $$\Alb (\hat{X}) = J \times  C'''_2,$$
where $ C'''_2$ is an elliptic curve.

  Hence we get again a morphism
$ \varphi \colon \hat{X} \ra (C'''_1 \times C'''_2)$ such that
$\varphi _*$ induces the above isomorphism $ \sN \cong  \Pi_{h_1} 
\times \Pi_{h_2}$,
and we can argue  exactly as before.

\smallskip

\underline{\bf  Case 00:} assume here that $h_1= h_2 = 1$.

Then the Albanese variety $A : = \Alb (\hat{X})$ is a surface, and we 
get a finite morphism
$f \colon (C''_1 \times C''_2) \ra A$.

Consider the induced injective homomorphism
$$ f^* \colon H^0 (\Omega^1_A) \ra  H^0 (\Omega^1_{C''_1 \times
C''_2} ) = H^0 (\Omega^1_{C''_1 } ) \oplus H^0 (\Omega^1_{C''_2} ).$$

This is compatible with wedge product, and since the image of  $H^0 
(\Omega^1_A)$
inside each $H^0 (\Omega^1_{C''_j } )$ is an isotropic subspace, it follows
that $H^0 (\Omega^1_A)$ splits as $\CC \omega_1 \oplus \CC \omega_2$
where, for $i\neq j$, $ f^* ( \omega_i) $ maps to zero in $H^0 
(\Omega^1_{C''_j })$.

We conclude that $A$ is isogenous to a product of elliptic curves, 
and again looking at
  $H_1 (\hat{X}, \ZZ)$ we derive that $A$ is a product of elliptic curves,
thereby showing that also in this case we get a morphism
$\varphi \colon \hat{X} \ra (C'''_1 \times C'''_2)$ to which we can 
apply the previous analysis.

We do not apply these considerations here
in order to study the  deformations of the surfaces $X$ that we construct.

The situation  can become rather involved, since $\varphi $ is not 
necessarily Galois.
We defer to \cite{keumnaie} for an interesting example, where 
$\varphi $ has degree 2.

In the case where $h_1 \neq 0$, $h_2 = 0$, we only have   $\varphi_1 
\colon \hat{X} \ra C'''_1 $
at our disposal, and the analysis seems much more difficult to handle.

\medskip

\section{The classification of standard isotrivial fibrations with
$p_g = q = 0$ where $X = (C_1 \times C_2)/G$ has rational double points}

In this section we will give a complete classification of the surfaces $S$
occurring as the minimal resolution of the singularities of a surface
$X:= (C_1
\times C_2)/G$, where $G$ is a finite group with an unmixed action on
a product of smooth
projective curves $C_1
\times C_2$ of respective genera $g_1, g_2 \geq 2$, and  such that
\begin{itemize}
\item[i)] $X$ has only {\em rational double points} as singularities,
\item[ii)] $p_g(S)=q(S)=0$.
\end{itemize}

     We denote by $\rho \colon S \rightarrow X$ the minimal resolution of
the singularities.

\begin{rem} 1) Note that the assumption $g_i \geq 2$ is equivalent to
$S$ being of
general type. In fact, $K_S=\rho^*K_X$ is nef, and
$K_S^2=K_X^2=\frac{8(g_1-1)(g_2-1)}{|G|}>0$, whence $S$ is  a minimal
surface of general
type.

\noindent 2) The analogous situations, with the assumption ii) replaced by
$p_g=q=1$, resp. $p_g=q=2$, have been classified by F. Polizzi in
\cite{pol}, resp. by M. Penegini in \cite{penego}.

\noindent 3) Recall that a cyclic quotient singularity is a
rational double
point if and only if it is of type $A_n$, for $n \in \mathbb{N}$.
\end{rem}

\subsection{The singularities of $X$} We shall show in the sequel
that, under the above hypotheses, $X$ has only ordinary
double points (i.e., singularities of type $A_1$). More precisely, we
shall derive the following improvement of
\cite[Prop. 4.1]{pol}.
\begin{prop}\label{onlynodes} Assume that $X:=C_1 \times C_2/G$ has
only rational double
points as singularities and that $\chi(\hol_X)=1$.

\noindent Then $X$ has $t:=8-K^2_X$ ordinary double points as
singularities. Moreover, $t\in\{0,2,4,6\}$.
\end{prop}

We use the following lemma (cf. \cite[Prop. 2.8]{pol2}).

\begin{lemma}\label{ftildesquare} Let $C_1$, $C_2$ be two compact
Riemann surfaces, and let $G$ be a finite group with an unmixed action
on $C_1 \times C_2$. Consider the surface $X:=C_1 \times C_2/G$, and
its minimal resolution of the singularities $\rho\colon S \rightarrow X$.

\noindent Let $F$ be a reduced fibre of the natural map $p_1\colon
X \rightarrow C_1/G$ as a Weil divisor, and let
$\tilde{F}:=\rho_*^{-1}F$ be its strict transform.

\noindent Assume that each singular point $p$ of $X$ on $F$ is of type
$A_{n(p)}$.

\noindent Then $$-\tilde{F}^2=\sum_{p\in F} \left(1-\frac{1}{n(p)+1}\right).$$
\end{lemma}

\Proof Set $q:=p_1(F)$. Then $\rho^*p_1^*(q)=k\tilde{F}+\sum_{E\in
Exc(\rho)} m_E E$ for some $k,m_E \in \NN$.

By assumption, for each $p \in F \cap \Sing(X)$, the exceptional divisor of
$\rho$ mapping to $p$ is union of $n(p)$ rational curves, say
$E_1,\ldots,E_{n(p)}$, with $E_i^2=-2$, $E_iE_{i+1}=1$, and $E_i \cap
E_j=\emptyset$ if $|i-j|>1$. By Serrano's theorem
(cf. \cite[(2.1)]{serrano}), we can assume $\tilde{F}E_i=0$ for each
$i<n(p)$ and
$\tilde{F}E_{n(p)}=1$.

Since $E_i \cdot \rho^*p_1^*(q)=0$, it follows that
$m_{E_i}=\frac{ik}{n(p)+1}$. Therefore
\begin{multline*}
  -\tilde{F}^2=\frac1k \tilde{F}(\rho^*p_1^*(q)-k\tilde{F})=\\
=\frac1k \tilde{F}\left(\sum_{E\in Exc(\rho)}m_E E\right)=
\frac1k \tilde{F}\left(\sum_{p \in F}\frac{n(p)k}{n(p)+1}\right).
\end{multline*}

\qed

It follows
\begin{cor}\label{onlynodesoracase} Assume that $X:=C_1 \times C_2/G$ has
only rational double
points as singularities and that $\chi(\hol_X)=1$.

\noindent Then
\begin{itemize}
\item[i)] either $X$ has $t:=8-K^2_X\in\{0,2,4,6\}$ ordinary double points as
singularities,
\item[ii)] or $X$ has three singular points, one of type $A_1$ and two
of type $A_3$.
\end{itemize}
\end{cor}

\Proof
By lemma \ref{ftildesquare}, considering all the fibres through the singular
points,
$\sum_{p\in Sing X} \left( 1-\frac1{n+1} \right) \in \NN,
$ or, equivalently
\begin{equation}\label{n/n+1inZ}
\sum_{p\in Sing X}  \frac1{n+1} \in \NN
\end{equation}

By \cite[prop. 4.1]{pol}, either
\begin{itemize}
\item[i)] $X$ has only nodes, or
\item[ii)] $X$ has two $A_3$ singularities, and at most two nodes.
\end{itemize}
Then formula (\ref{n/n+1inZ}) shows that in case i) the number of nodes is even
whereas in case ii) it is odd.

Note that (cf. \ref{eulernodes})
\begin{equation*} K^2_S =
K^2_X=\frac{8(g(C_1)-1)(g(C_2)-1)}{|G|},  \ \ \
e(S)=\frac{K_S^2+3t}2.
\end{equation*}
By Noether's formula
$$1=\chi(\hol_S)=
\frac{1}{12}\left(K_S^2+e(S)\right)=
\frac{3K_S^2+3t}{12 \cdot 2}=\frac{K_S^2+t}{8},
$$ and it follows that $K_S^2=8-t$. Since
$K_S^2>0$, $t\leq 6$.

\qed

\begin{rem}
1) In a forthcoming paper (\cite{bp}) it is shown that case ii) of
corollary \ref{onlynodesoracase} do not occur.

2) The case $t=0$ (i.e.,  $G$ acts freely on $C_1 \times C_2$) has been
completely classified in \cite{bcg}. Therefore in what follows we
shall assume $t>0$.
\end{rem}

\subsection{The signatures} We know that $X$ determines the following data
\begin{itemize}
\item[-] a finite group $G$,
\item[-] two orbifold surface groups $\BT_1:=\BT(g'_1;m_1,\ldots, m_r)$,
$\BT_2:=\BT(g'_2;n_1,\ldots, n_s)$,
\item[-] two appropriate homomorphisms $\varphi_i \colon \BT_i \rightarrow
G$, such that the
stabilizers fulfil certain conditions ensuring that $X$ has only
$2,4$ or $6$ nodes
respectively.
\end{itemize} The assumption $q=0$ implies that $C_i/G$ is rational,
i.e., $g'_1 = g'_2 =
0$, whence $\BT_i$ is indeed a polygonal group for $ i=1,2$.

\medskip\noindent We will use the combinatorial restriction forced by
the assumption
$\chi(S) = 1$, and the conditions on the singularities of $X$ in
order to determine the
possible signatures $T_1 = (m_1, \ldots , m_r)$, $T_2 = (n_1, \ldots
, n_s)$ of the
respective polygonal groups.

\noindent We associate to an $r$-tuple $T_1$ the following numbers:
\begin{equation}\label{theta alpha}
\Theta_1:=-2 +\sum_{i=1}^r \left( 1 - \frac{1}{m_i} \right), \ \ \ \ \ \
\alpha_1:=\frac{K_S^2}{4\Theta_1}
\end{equation} and similarly we associate $\Theta_2$ and $\alpha_2$ to $T_2$.

The geometric meaning of $\alpha_i$ is given by Hurwitz' formula.
\begin{lemma}\label{alpha-g}
$\alpha_1=g_2-1$, $\alpha_2=g_1-1$, and $|G|=\frac{8\alpha_1\alpha_2}{K_S^2}$.
\end{lemma}
\Proof Hurwitz' formula applied to $p_i$ gives
$$2(g_i-1)=|G|\Theta_i.$$
     Therefore by (\ref{K2})and (\ref{theta alpha}), we obtain
$$\alpha_1=\frac{K^2_S}{4\Theta_1}=\frac{2(g_1-1)(g_2-1)}{|G|\Theta_1}=g_2-1,$$
and
similarly $\alpha_2=g_1-1$. In particular,
     $|G|=\frac{8\alpha_1\alpha_2}{K_S^2}$.

\qed

\begin{rem} Note that the above lemma implies that, $\forall i,\
\alpha_i\in \NN$. As we
shall see, this is a strong restriction on the set of possible signatures.

\end{rem}

\begin{lemma} Let $T_1=(m_1,\ldots, m_r)$ be the signature of one of
the polygonal groups
yielding a surface with $t$ nodes. Then each $m_i$ divides
$2\alpha_1$. Moreover, there are at most $\frac{t}{2}$ indices
$i \in \{ 1, \dots , r\}$  such that $m_i$ does not divide $\alpha_1$.
\end{lemma}

\Proof Consider $p_1 \colon X \rightarrow C'_1 : = C_1/G$.
Fix $m_i$ in $T_1$ and
let
$h_i=\varphi_1(c_i)$ where the $c_i \in \BT_1$ are generators as in
(\ref{polygroup}).
Then
$h_i$ determines a
point $q_i \in C_1/G$, image of the points stabilized by $h_i$. Now,
$p_1^* (q_i)=m_iW_i$
for some irreducible Weil divisor
$W_i$. It follows that  the intersection number of $K_X$ with a fibre
$F$, which
equals on one side $2g_2-2=2\alpha_1$, can be written as $m_i W_i \cdot K_X$.
We conclude since $K_X$ is Cartier, whence $W_i \cdot K_X$ is an integer (cf.
\cite{fulton}).

\medskip\noindent If $C_2 \rightarrow C_2/h_i\cong W_i$ is \'etale,
then, by Hurwitz'
formula, $m_i$ divides $\alpha_1$.

\noindent Else  $h_i$ has order $= 2d$ and its $d$-th power
stabilizes some points of
$C_2$, and therefore $W_i$ passes through some nodes of $X$.  By
lemma (\ref{ftildesquare}) (for $F=W_i$), since $\tilde{F}^2\in \ZZ$
there are at least
two nodes on $W_i$, whence this can happen for at most $\frac{t}2$ indices $i$.

\qed

\medskip\noindent  Given $t \in \{2,4,6\}$ we want now to determine
all $r$-tuples
$T:=(m_1,\ldots, m_r)$ of positive integers which fulfill the
numerical conditions we
just found {\it i.e.},
\begin{itemize}
\item[i)] $\alpha : = \frac{8 -t}{4\Theta(m_1, \dots, m_r)}$ (cf.
(\ref{theta alpha})) is
a strictly positive integer;
\item[ii)] each $m_i$ divides $2\alpha$;
\item[iii)] at most $\frac{t}{2}$ of the $m_i$'s do not divide $\alpha$.
\end{itemize}

\noindent In other words we are looking for a priori possible
signatures of polygonal
groups giving rise to standard isotrivial fibrations with $p_g = q=
0$, and $t$ nodes.

\noindent In order to write a MAGMA script to list all the possible candidates
for the signatures,
we still have to give effective bounds for the numbers $r$ and $m_i$.
The following lemma
will do this.

\begin{lemma}\label{boundrm} Let $t \in \{2,4,6\}$ and assume that
$(m_1, \dots , m_r)$ is
a sequence of integers
$m_i \geq 2$ fulfilling conditions i), and ii) above. Then we have:

1) $r \leq 7$ ({\em i.e.}, $p_1$ (resp. $p_2$) has at most $7$ branch points);

2) {\rm for all}  $i $ we have $\ m_i \leq 3(10-t)$.
\end{lemma}
\Proof 1) If $r\geq 8$, then $\Theta \geq -2+\frac82=2$. Then $1 \leq
       \alpha=\frac{8-t}{4\Theta}\leq \frac{6}{4 \cdot 2}$, a contradiction.

\noindent 2) We can assume $m_1 \geq m_i$ for all $i$. We want to
show that $m_1\leq
3(10-t)$.

\noindent Note that,
since  $\alpha$ (and therefore also
$\Theta$) is positive,  $r \geq 3$ and if $r=3$ at most one $m_i$ can be
equal to $2$. Therefore it holds
$$\Theta+\frac{1}{m_1}= -1 +\sum_{i=2}^r (1 - \frac{1}{m_i} ) \geq
\frac16,$$ which is equivalent to
$6(\Theta m_1+1) \geq m_1$. Since, by condition ii), $m_1
\leq 2\alpha$, we
conclude
$$ m_1 \leq 6(2\alpha \Theta+1)=6\left(\frac{8-t}{2}+1\right)
$$

\qed

\medskip\noindent  The MAGMA script {\bf ListOfTypes}, which can be
found in the appendix as all other scripts, lists all signatures
fulfilling  conditions i), ii), iii): see table \ref{typespr}.

\begin{table}[ht]
\caption{The signatures which respect conditions i), ii) and iii)}
\label{typespr}
\begin{tabular}{|c|c|c||c|c|c||c|c|c|}
\hline
$K^2$&T&$\alpha$&$K^2$&T&$\alpha$&$K^2$&T&$\alpha$\\
\hline\hline 2&  (2,3,7)&21&4&    (2,3,7)&42&6&     (2,3,7)&63\\ 2&
(2,3,8)&12&4&
(2,3,8)&24&6&     (2,3,8)&36\\ 2&  (2,3,9)& 9&4&    (2,3,9)&18&6&
(2,3,9)&27\\ 2&
(2,3,12)& 6&4&   (2,3,10)&15&6&    (2,3,12)&18\\ 2&  (2,4,5)&10&4&
(2,3,12)&12&6&
(2,3,15)&15\\ 2&  (2,4,6)& 6&4&   (2,3,18)& 9&6&    (2,3,24)&12\\ 2&
(2,4,8)& 4&4&
(2,4,5)&20&6&     (2,4,5)&30\\ 2&  (2,5,5)& 5&4&    (2,4,6)&12&6&
(2,4,6)&18\\ 2&
(2,6,6)& 3&4&    (2,4,8)& 8&6&     (2,4,7)&14\\ 2&  (3,3,4)& 6&4&
(2,4,12)& 6&6&
(2,4,8)&12\\ 2&  (3,3,6)& 3&4&    (2,5,5)&10&6&    (2,4,10)&10\\ 2&
(4,4,4)& 2&4&
(2,5,10)& 5&6&    (2,4,16)& 8\\ 2&(2,2,2,3)& 3&4&    (2,6,6)& 6&6&
(2,5,5)&15\\
2&(2,2,2,4)& 2&4&    (2,8,8)& 4&6&    (2,6,12)& 6\\
     &         &  &4&    (3,3,4)&12&6&     (2,7,7)& 7\\
     &         &  &4&    (3,3,6)& 6&6&     (3,3,4)&18\\
     &         &  &4&    (3,4,4)& 6&6&     (3,3,6)& 9\\
     &         &  &4&    (3,6,6)& 3&6&    (3,3,12)& 6\\
     &         &  &4&    (4,4,4)& 4&6&     (3,4,6)& 6\\
     &         &  &4&  (2,2,2,3)& 6&6&     (4,4,8)& 4\\
     &         &  &4&  (2,2,2,4)& 4&6&   (2,2,2,4)& 6\\
     &         &  &4&  (2,2,3,3)& 3&6&   (2,2,2,8)& 4\\
     &         &  &4&  (2,2,4,4)& 2&6&   (2,3,3,3)& 3\\
     &         &  &4&(2,2,2,2,2)& 2&6& (2,2,2,2,4)& 2\\
\hline
\end{tabular}
\end{table}

%\begin{rem}The code {\bf ListOfTypes} is in the appendix together
%with all the other scripts we wrote.
%The reader should know that, in order to simplify
%the programming, we
%have not considered a signature as a sequence of variable length
%$(m_1, \ldots, m_r)$,
%$r\leq 7$, $m_i\geq 2$, but {\em as a sequence
%$(m_1, \ldots, m_7)$, $m_i\geq 1$, by adding the necessary number of
%$1$'s}.
%
%\noindent So, for example, $(2,3,7)$ becomes $(1,1,1,1,2,3,7)$.
%\end{rem}

\subsection{The possible groups $G$}

For each surface $X=(C_1 \times C_2)/G$ with $p_g = q = 0$ and nodes
as above we have two associated signatures $T_1$ and
$T_2$, such that $G$ is a quotient of the corresponding polygonal groups.

     Moreover, as
shown by lemma \ref{alpha-g}, the order of the group depends only on $T_1$,
$T_2$ and $K^2 : = K^2_X = 8-t$. Having found, for each possible
$K^2$, a finite list of possible signatures, we have then only
finitely many groups to consider.

%We proceed now using three MAGMA functions. The first one, {\bf
%ElsOfOrd}, lists
%    all elements of a given order in a finite group. The second one,
%{\bf ExistSphericalGenerators},  determines whether a finite group is
%a quotient of a
%polygonal group of given signature.

We can then write a script which  computes, for
each $K^2$, all possible triples $(T_1,T_2,G)$, where $G$ is a
group of order $|G|=\frac{8\alpha_1\alpha_2}{K^2}$, which
is a quotient of both  polygonal groups of respective signatures
$T_1, T_2$. This is done by the script  {\bf ListGroups}.

Note tha our code skips those pairs of signatures
giving rise to groups of order 1024 or bigger than 2000, since these
cases are not covered by the MAGMA SmallGroup database of finite groups.

Moreover, the code skips the case $|G|=1152$, since there are more than
$10^6$ groups of this
order and this causes extreme computational complexity.

\noindent We shall see now, that the cases skipped by the MAGMA 
script actually do not occur.

\noindent
\underline{$|G|=1024$:} For each pair of
signatures $T_1$, $T_2$ in the same column of the table
\ref{typespr}, $\frac{8\alpha_1\alpha_2}{K^2} \neq 1024$. So this
case does not occur.

\noindent
\underline{$|G|=1152$:} Looking again at  table
\ref{typespr}, this can happen
only in one case:
\begin{itemize}
     \item[-] $K^2=4$, $T_1=T_2=(2,3,8)$.
\end{itemize} Since the abelianization of $\BT(2,3,8)$ is
$\ZZ_2$, the abelianization of $G$ can have order at most $2$. The
MAGMA smallgroup
database shows that there are exactly $44$ groups of order $1152$ with
abelianization of order at most
$2$. It can now easily be checked that
none of these groups
is a quotient of
$\BT(0;2,3,8)$.

\noindent
\underline{$|G|>2000$:} This can happen in $6$ cases, all with $T_1=(2,3,7)$.

\noindent More precisely we have:
\begin{itemize}
     \item[-] $K^2=4$, $T_2=(2,3,7)$ or $(2,3,8)$, or
\item[-]  $K^2=6$, and $T_2$ one of the following: $(2,3,7)$,
$(2,3,8)$, $(2,3,9)$,
$(2,4,5)$.
\end{itemize}

\noindent Since each quotient of $\BT(0;2,3,7)$ is a perfect group we
can look at the
MAGMA database of perfect groups, which is complete up to order 50000.
Table \ref{typespr} gives that
$|G|$, in the six respective cases, is one of the numbers $3528$,
$2016$, $5292$, $3024$,
$2268$, and $2520$. But there is only one perfect group having order
equal to one of these
numbers, namely, $\mathfrak{A}_7$;  a quick verification shows that
$\mathfrak{A}_7$ is
not a quotient of
$\BT(0;2,3,7)$.

\subsection{The classification}\label{explainthetable}

We have now, for each value $K^2 = 2,4,6$, a finite list of triples
$(T_1,T_2,G)$ such
that $T_1$ and $T_2$ fulfill the numerical conditions i), ii), iii),
$G$ has the order
given by lemma \ref{alpha-g} and is a quotient of both polygonal
groups $\BT_1$ and
$\BT_2$ associated to the signatures $T_1$ and $T_2$.

To each of those triples correspond many families of surfaces, one
for each pair of
appropriate homomorphisms $\BT_1 \rightarrow G$, $\BT_2 \rightarrow G$.

     We construct these families varying (here $C'_1 = C'_2  = \PP^1 $)
     the branch point sets $\sB_1 : = \{p_1, \dots p_r \} \subset \PP^1$,
$\sB_2 : = \{q_1, \dots q_s \} \subset \PP^1$,
and choosing respective standard bases for
     the fundamental groups  $\pi_1 (\PP^1 \setminus \sB_i)$.
Then we use
Riemann's existence theorem to construct the two curves
$C_1$ and $C_2$ with an action of $G$ of respective signatures $T_1$ and
$T_2$, we finally consider $X= (C_1 \times C_2)/G$.

Observe that, since $  C'_i  = \PP^1 $, any appropriate homomorphism
of $\BT_i$ to $G$ is completely determined by a so called {\em
spherical system of
generators  } of $G$ of signature $T_i$, i.e., in the case of
$\BT_1$, a sequence of
elements
$h_1,
\dots , h_r$ which generate $G$, satisfy $h_1 \cdot   \ldots \cdot h_r =1$,
and fulfill moreover $ ord (h_i) = m_i$.

\noindent It turns out that most of these surfaces have too many or
too bad singularities,
and we only take into consideration the surfaces whose singular locus
consists of exactly
$t = 8-K^2$ nodes.

\noindent However, different pairs of appropriate homomorphisms can
yield the same
family of surfaces, due to different choices of a standard basis: this is
takien into account by declaring that
two pairs of  appropriate homomorphisms are equivalent iff they are in
the same orbit
of the
natural action of
$\Aut(G)$ (simultaneosly on both elements of the pair) and of the 
respective braid groups (the
second equivalence
relation is generated by the so-called {\em Hurwitz moves}, cf.
\cite{bacat}).

\medskip\noindent We determine our equivalence classes by using two
MAGMA scripts (and many subroutines).

%\noindent {\bf FSGUpToConjugation} finds all possible systems of
%spherical generators
%of a group $G$ which are of the  prescribed signature (first modulo
%inner automorphisms, in order to reduce the memory load).

%\noindent {\bf CheckSings} uses {\bf FSGUpToConjugation} to produce
%all possible pairs of
%spherical generators of a group which are of the  prescribed
%signatures. It then proceeds
%to verify  whether such  pairs yields a surface with the right
%number of nodes and no further singularities, interrupting the
%computation as soon as it
%finds a pair, which fulfills these conditions.

\noindent {\bf ExistingNodalSurfaces} gives as output, for each
$K^2$, all possible triples $(T_1,T_2,G)$ which yield at least one 
nodal surface
with the right number of
nodes.
%  simply by running {\bf CheckSings} on the output of ListGroups.
It gives $7$ possible triples ($T_1$, $T_2$, $G$) in the case $K^2=2$,
$11$ possible triples in the case $K^2=4$ and $6$ possible triples
in the case $K^2=6$.

These results are listed in table \ref{surfaces} and, more
precisely, in the first 6
columns, showing respectively $K^2$, $T_1$, $T_2$, $g_1$,
$g_2$ and the group $G$.

The corresponding families of surfaces have dimension equal to
     $ r + s - 6 = \# T_1+\# T_2-6$.

To find all the irreducible families corresponding to a given triple
we write the script {\bf FindAllComponents}
which determines all possible equivalence classes for pairs of spherical
systems of generators, producing one representative for each class.

%Subroutines are: {\bf AutGr}, which describes the group of
%automorphisms of $G$ (as
%a concrete set), {\bf HurwitzMove}, which performs  Hurwitz moves on
%a sequence of
%elements of a group, and  {\bf HurwitzOrbit} which determines the
%whole orbit for the
%action of {\bf HurwitzMove} on a sequence of elements of a group, and
%gives as outputs (
%in order to
%save on memory) sequences of elements  whose corresponding sequence of orders
%is non decreasing.

%We wrote moreover: {\bf SphGens}, which lists all possible sets of
%spherical generators of
%a group of prescribed types, and {\bf CheckSingsEl}, which checks if
%two given sequences
%of spherical generators of a group give exactly the prescribed
%singularities. Finally,
%{\bf FindAllComponents}, considers for each group and pair of
%signatures  all the possible
%pairs of systems spherical generators of the given signatures,
%partitions them in orbits
%for the  given equivalence relation by using {\bf HurwitzOrbit} and
%{\bf AutGr}, saves
%only one representative for each orbit and returns it, if it passes
%the singularity test
%{\bf CheckSingsEl}. \footnote{For the reader who wishes to have a
%look at the scripts:
%since we found that $r,s \leq 5$, the signatures are now written as
%sequences of
%$5$ numbers (instead of $7$ as before).  For instance, $(2,3,7)$ is now
%written as $(1,1,2,3,7)$.}

Running {\bf FindAllComponents} on the $24$ triples given by {\bf
ExistingNodalSurfaces},
we have find one family in $21$ cases and two families in the
remaining three cases. The number of families found is written in the
seventh column of the table \ref{surfaces}.
\noindent Finally, we compute a presentation of the fundamental group
of each of the constructed  surfaces. To do this we start from the given
presentation of the direct product of the two given polygonal groups.
Using the Reidemeister - Schreier routine, we compute a presentation
of its finite index subgroup $\HH$, and divide out the appropriate
torsion elements in $\HH$ to obtain a presentation of $\pi_1$. Since
we know by theorem \ref{strfund}, that $\pi_1$ is surface times
surface by finite, we go through all the normal subgroups of finite index
until such a group appears.

Let us point out that, in the three cases in which there are two
families, they both have the same fundamental group. Therefore we
listed the fundamental
groups in the ninth column of the table \ref{surfaces} without
distinguishing the two cases.

\noindent In the eight column we have listed $H_1(X,\ZZ)$, which is
always finite.

\begin{rem}\label{---}
The description of the groups in the table
\ref{surfaces} requires some explanation.

\noindent As usual, $\ZZ_d$ is the cyclic group of order $d$, $\ZZ$
the cyclic group of
infinite order,  $\mathfrak{S}_n$ is the symmetric group in $n$
letters, $\mathfrak{A}_n$
its only index $2$ subgroup.

\noindent
$PSL(2,7)$ is the group of $2 \times 2$ matrices over
${\mathbb F}_7$ with determinant $1$ modulo the subgroup generated by $-Id$.

\noindent
$D_{p,q,r}=\langle x,y|x^p,y^q,xyx^{-1}y^{-r} \rangle$, and
$D_n=D_{2,n,-1}$ is the usual dihedral group of order $2n$.

$G(32,2)$ is the second group of order $32$ in the MAGMA database, a
$2$-group of
exponent $4$ whose elements of order $2$ form the center,  giving a
central extension $1
\rightarrow \ZZ_2^3 \rightarrow \pi_1 \rightarrow \ZZ_2^2 \rightarrow
1$, which does not split.

Finally, we have  semidirect products $H \rtimes \ZZ_r$; to specify
them, we have to indicate the image of the generator of $\ZZ_r$ in
$Aut(H)$.

For $H = \ZZ^2$ either $r$ is even, and then the image of the 
generator of $\ZZ_r$ in
$Aut(H)$ is $-Id$, or
$r=3$ and the image of the generator of $\ZZ_3$ is the matrix $\begin{pmatrix}
-1&-1\\
1&0
\end{pmatrix}$.

Else $H$ is finite and $r=2$;
for
$H=\ZZ_3^2$ the image of the generator of $H$ is
$-Id$; for $H=\ZZ_2^4$ it is
$\begin{pmatrix} 1&0\\ 1&1
\end{pmatrix} \oplus \begin{pmatrix} 1&0\\ 1&1
\end{pmatrix}.$
\end{rem}

Except for the last two cases, all the groups appearing as fundamental
groups are polycyclic by finite.
The groups in line 10 and 16 of table \ref{surfaces} have the same
commutator quotient and also the same commutator subgroup,
but they are not isomorphic since, for example, the one in line 16 has
no normal subgroups of index $4$ with free abelianization.
The groups in the lines 8 and 14 are isomorphic to each other, as the 
ones in the  lines
10 and 13, as well as the groups in the lines 17
and 18 (they are isomorphic to the fundamental groups of the Keum 
Naie surfaces, \cite{keumnaie}).

We observe here that the isomorphism problem
for polycyclic by finite groups is decidable (\cite{Segal}).

\subsection{Bloch's conjecture}

Another important feature of surfaces with $p_g = 0$ is the following

\begin{conj}
Let $S$ be a smooth surface with $p_g = 0$ and let $A_0(S) = 
\oplus_{i=0}^{\infty} A_0^i(S)$ be the group of rational equivalence 
classes of zero cycles on $S$. Then the kernel $T(S)$ of the natural 
morphism $A_0^0(S) \ra \Alb(S)$ is trivial.
\end{conj}

The conjecture has been proven for $\kappa(S) <2$ by Bloch, Kas and 
Liebermann (cf. \cite{bkl}). If instead $S$ is of generaltype, then 
$q(S) = 0$, whence Bloch's conjecture asserts for those surfaces that 
$A_0(S) \cong \ZZ$.

Inspite of the efforts of many authors, there are only few cases of 
surfaces of general type for which Bloch's conjecture has been 
verified (cf. e.g. \cite{inosemik}, \cite{barlowbloch}, \cite{keum}, 
\cite{voisin}).

Recently S. Kimura introduced the following notion of {\em finite 
dimensionality} of motives (\cite{kimura}).
\begin{definition}
  Let $M$ be a motive.

  Then $M$ is {\em evenly finite dimensional} if there is a natural 
number $n \geq 1 $ such that $\wedge ^n M = 0$.

  $M$ is {\em oddly finite dimensional} if there is a natural number 
$n \geq 1 $ such that $\Sym ^n M = 0$.

  And, finally, $M$ is {\em finite dimensional} if $M = M^+ \oplus 
M^-$, where $M^+$ is evenly finite dimensional and $M^-$ is oddly 
finite dimensional.
\end{definition}

Using this notation, he proves the following
\begin{theo}\label{kimura}
1) The motive of a smooth projective curve is finite dimensional 
(\cite{kimura}, cor. 4.4.).

2) The product of finite dimensional motives is finite dimensional 
(loc. cit., cor. 5.11.).

3) Let $f \colon M \ra N$ be a surjective morphism of motives, and 
assume that $M$ is finite dimensional. Then $N$ is finite dimensional 
(loc. cit., prop. 6.9.).

4) Let $S$ be a surface with $p_g = 0$ and suppose that the Chow 
motive of $X$ is finite dimensional. Then $T(S) = 0$ (loc.cit., cor. 
7.7.).
\end{theo}

Using the above results we  obtain
\begin{theo}
  Let $S$ be the minimal model of a product-quotient surface with $p_g = 0$.

   Then Bloch's conjecture holds for $S$, namely, $A_0(S) \cong \ZZ$.
\end{theo}

\begin{proof}
  Let $S$ be a product-quotient surface.

  Then $S$ is the minimal model of $X = (C_1 \times C_2)/G$. Since $X$ 
has rational singularities $T(X) = T(S)$.

  By thm. \ref{kimura}, 2), 3) we have that the motive of $X$ is 
finite dimensional, whence, by 4),  $T(S) = T(X) = 0$.

  Since $S$ is of general type we have also $ q(S)=0$, hence $A_0^0(S) 
= T(S) = 0$.

\end{proof}

\begin{cor}
  All the surfaces in table \ref{surfaces} and all the surfaces in 
\cite{bacat}, \cite{bcg} satisfy Bloch's conjecture.
\end{cor}

%%%%%%%%%%%%%%%%%%%%%%%%%%%%%%%%%%%%%%%%%%%%%
\section{The surfaces}\label{thedescription}
%%%%%%%%%%%%%%%%%%%%%%%%%%%%%%%%%%%%%%%%%%%%%
This section is an expanded version of table \ref{surfaces}.
In the sequel we will follow the scheme below:

\begin{itemize}
\item[$G$:] here we write the Galois group $G$ (most of the times as 
permutation group);
\item[$T_i$:] here we specify the respective types of the pair of 
spherical generators of the group $G$;
\item[$S_1$:] here we list the first set of spherical generators;
\item[$S_2$:]  here we list the second set of spherical generators;
\item[$H_1$:] the first homology group of the surface;
\item[$\pi_1$:] the fundamental group of the surface;
\item[$\HH$:] the generators of the preimage $\HH$ of the diagonal of 
$G \times G$;
\item[] and their images in the fundamental group $\HH / \Tors(\HH)$.
\end{itemize}

\subsection{$K^2=2$, Galois group $PSL(2,7)$:}
\begin{itemize}
\item[$G$:] $\langle (34)(56),(123)(457) \rangle < {\mathfrak S}_7$
\item[$T_i$:] $(2,3,7)$ and $(4, 4, 4)$
\item[$S_1$:] (13)(26), (127)(345), (1762354)
\item[$S_2$:] (1632)(47), (1524)(36), (1743)(25)
\item[$H_1$:] $\ZZ_2^2$
\item[$\pi_1$:] $\ZZ_2^2$
\item[$\HH$:] $\HH_1 :=c_1c_3^2c_2d_2$, $\HH_2 := c_2c_3^{-1}c_1c_3d_3^{-1}$
\item[]$\HH_1 \mapsto (\overline{1},\overline{0})$, $\HH_2 \mapsto 
(\overline{0},\overline{1})$
\end{itemize}

\subsection{$K^2=2$, Galois group $PSL(2,7)$:}
\begin{itemize}
\item[$G$:] $\langle(34)(56),(123)(457) \rangle < {\mathfrak S}_7$
\item[$T_i$:] $(2,3,7)$ and $(4, 4, 4)$
\item[$S_1$:](13)(26),(127)(345),(1762354)
\item[$S_2$:](16)(2537),(1734)(26),(1452)(67)
\item[$H_1$:] $\ZZ_2^2$
\item[$\pi_1$:] $\ZZ_2^2$
\item[$\HH$:] $\HH_1 :=c_2c_1d_3^{-1}d_2d_1^{-2}$, $\HH_2 
:=c_3c_2c_1d_1d_3^{-1}d_1^{-1}$
\item[] $\HH_1 \mapsto (\overline{1},\overline{0})$, $\HH_2 \mapsto 
(\overline{0},\overline{1})$
\end{itemize}

\subsection{$K^2=2$, Galois group ${\mathfrak S}_5$:}
\begin{itemize}
\item[$G$:] ${\mathfrak S}_5$
\item[$T_i$:] $(2,4,5)$ and $(2,6,6)$
\item[$S_1$:] (25),(1435),(12534)
\item[$S_2$:] (13)(25),(142)(35),(15)(243)
\item[$H_1:$] $\ZZ_3$
\item[$\pi_1:$] $\ZZ_3$
\item[$\HH$:] $\HH_1 :=c_1c_2^{-1}c_3c_2^{-1}d_2$, $\HH_2
:=c_3^2c_2^{-1}c_3^{-1}d_1d_2^{-1}$, $\HH_3:=c_3c_2c_3^{-1}c_1d_3^{-2}$
\item[] $\HH_1 \mapsto \overline{1}$ and $\HH_2$, $\HH_3 \mapsto\overline{2}$
\end{itemize}

\subsection{$K^2=2$, Galois group ${\mathfrak A}_5$:}

\begin{itemize}
\item[$G$:] $\mathfrak{A}_5$
\item[$T_i$:] $(2,2,2,3)$ and $(2,5,5)$
\item[$S_1$:] (14)(35), (15)(24), (13)(24), (154)
\item[$S_2$:] (23)(45), (15342), (13425)
\item[$H_1$:] $\ZZ_5$
\item[$\pi_1$:] $\ZZ_5$
\item[$\HH$:] $\HH_1 := c_1c_3d_2^{-1}$, $\HH_2 := c_2c_3d_2^{-1}d_3$
\item[] $\HH_1 \mapsto \overline{1}$, $\HH_2 \mapsto \overline{2}$
\end{itemize}

\subsection{$K^2=2$, Galois group ${\mathfrak S}_4 \times \ZZ_2$:}

\begin{itemize}
\item[$G$:] $\langle (12), (13), (14), (56) \rangle < \mathfrak{S}_6$
\item[$T_i$:] $(2,2,2,4)$ and $(2,4,6)$
\item[$S_1$:] (12)(56), (34), (14), (1342)(56)
\item[$S_2$:] (13)(56), (1342), (124)(56)
\item[$H_1$:] $\ZZ_2^2$
\item[$\pi_1$:] $\ZZ_2^2$
\item[$\HH$:] $\HH_1 := c_1c_3d_3$, $\HH_2 := c_1c_4^{-1}d_1d_3d_2$, 
$\HH_3 := c_3d_2d_1d_2^{-1}d_3$
\item[] $\HH_1$, $\HH_3 \mapsto (\overline{1},\overline{0})$, $\HH_2 \mapsto
(\overline{0},\overline{1})$
\end{itemize}

\subsection{$K^2=2$, Galois group ${\mathfrak S}_3 \times {\mathfrak S}_3$:}

\begin{itemize}
\item[$G$:] $\langle (12), (13), (45), (46) \rangle < \mathfrak{S}_6$
\item[$T_i$:] $(2,2,2,3)$ and $(2,6,6)$
\item[$S_1$:] (13),(46),(23)(56),(132)(465)
\item[$S_2$:] (23)(46),(123)(45),(12)(456)
\item[$H_1$:] $\ZZ_3$
\item[$\pi_1$:] $\ZZ_3$
\item[$\HH$:] $\HH_1 := c_2c_1d_2^{-1}d_3$, $\HH_2 := 
c_3c_2d_1d_2^{-3}$, $\HH_3 :=
c_4c_1d_2^2d_1d_2^{-1}$
\item[] $\HH_1$, $\HH_3 \mapsto \overline{1}$ and $\HH_2 \mapsto \overline{0}$.
\end{itemize}

\subsection{$K^2=2$, Galois group $\ZZ_4^2$:}

\begin{itemize}
\item[$G$:] $\ZZ_4^2$
\item[$T_i$:] $(4,4,4)$ and $(4,4,4)$
\item[$S_1$:] 
$(\overline{1},\overline{3}),(\overline{1},\overline{0}),(\overline{2},\overline{1})$
\item[$S_2$:] 
$(\overline{3},\overline{2}),(\overline{0},\overline{1}),(\overline{1},\overline{1})$
\item[$H_1$:] $\ZZ_2^3$
\item[$\pi_1$:] $\ZZ_2^3$
\item[$\HH$:] $\HH_1 := c_1d_2d_1^{-1}$, $\HH_2 := c_2c_1^{-1}d_2$, 
$\HH_3 := c_1d_1^{-1}d_2$
\item[] $\HH_1 \mapsto (\overline{1},\overline{0},\overline{0})$, 
$\HH_2 \mapsto
(\overline{0},\overline{1},\overline{0})$ , $\HH_3 \mapsto
(\overline{0},\overline{0},\overline{1})$
\end{itemize}

\subsection{$K^2=2$, Galois group $D_4 \times \ZZ_2$:}

\begin{itemize}
\item[$G$:] $\langle (1234), (14)(23), (56) \rangle < {\mathfrak S}_6$
\item[$T_i$:] $(2,2,2,4)$ and $(2,2,2,4)$
\item[$S_1$:] (24),(56),(12)(34),(1432)(56)
\item[$S_2$:] (12)(34),(13)(56),(13)(24)(56),(1234)
\item[$H_1$:] $\ZZ_4 \times \ZZ_2$
\item[$\pi_1$:] $\ZZ_4 \times \ZZ_2$
\item[$\HH$:] $\HH_1 := c_1c_3d_4$, $\HH_2 := c_1d_1d_4^{-1}$, $\HH_3 
:= c_4d_3d_4$,
$H_4 := c_1c_4^{-1}c_1d_3d_4$
\item[] $\HH_1 \mapsto (\overline{1},\overline{0})$, $\HH_2\mapsto 
(\overline{3},\overline{0})$,
$\HH_3$, $\HH_4 \mapsto(\overline{1},\overline{1})$
\end{itemize}

\subsection{$K^2=4$, Galois group ${\mathfrak S}_5$:}
\begin{itemize}
\item[$G:$] ${\mathfrak S}_5$
\item[$T_i$:] $(2,4,5)$ and $(3,6,6)$
\item[$S_1$:] (25), (1435), (12534)
\item[$S_2$:] (132), (135)(24), (15)(243)
\item[$H_1$:] $\ZZ_2^3$
\item[$\pi_1$:] $\ZZ^2
\rtimes_\varphi \ZZ_3$, $\varphi(\overline{1})=
\begin{pmatrix}
0&1\\
-1&-1
\end{pmatrix}
$.
\item[$\HH$:] $\HH_1 := c_3^2c_2^{-1}d_2$, $\HH_2 := 
c_3^{-1}c_2c_3^{-1}d_3^{-1}$
\item[] $\HH_1 \mapsto \overline{1}$, $\HH_2 \mapsto (-1,0)\cdot \overline{2}$
\end{itemize}

\subsection{$K^2=4$, Galois group ${\mathfrak A}_5$:}

\begin{itemize}
\item[$G$:] $\mathfrak{A}_5$
\item[$T_i$:] $(2,2,3,3)$ and $(2,5,5)$
\item[$S_1$:] (15)(34), (12)(35), (354), (125)
\item[$S_2$:] (23)(45), (15342), (13425)
\item[$H_1$:] $\ZZ_{15}$
\item[$\pi_1$:] $\ZZ_{15}$
\item[$\HH$:] $\HH_1 := c_4d_2d_3^{-1}$, $\HH_2 := 
c_3^{-1}d_2^2d_3^{-1}$, $\HH_3 :=
c_3c_1d_2^{-1}d_3$
\item[] $\HH_1 \mapsto \overline{1}$, $\HH_2 \mapsto \overline{4}$, 
$\HH_3 \mapsto \overline{14}$
\end{itemize}

\subsection{$K^2=4$, Galois group ${\mathfrak S}_4 \times \ZZ_2$:}

\begin{itemize}
\item[$G$:] $\langle (12), (13), (14), (56) \rangle < \mathfrak{S}_6$
\item[$T_i$:] $(2,2,4,4)$ and $(2,4,6)$
\item[$S_1$:] (24), (24), (1324)(56), (1423)(56)
\item[$S_2$:] (13)(56), (1342), (124)(56)
\item[$H_1$:] $\ZZ_2^2 \times \ZZ_4$
\item[$\pi_1$:] $\ZZ^2
\rtimes_{\varphi} \ZZ_4$, $\varphi(\overline{1}) = -Id$
\item[$\HH$:] $\HH_1 := c_1c_4^{-1}c_1d_1d_3^2$,
$\HH_2 := c_3c_1d_3^{-1}d_1d_2^{-1}d_3$,
$\HH_3 := c_1c_3^{-1}c_1d_3^{-2}d_1$
\item[]
$\HH_1 \mapsto \overline{1}$, $\HH_2 \mapsto (1,0)$,
   $\HH_3 \mapsto (0,1)\cdot \overline{1}$
\end{itemize}

\subsection{$K^2=4$, Galois group ${\mathfrak S}_4 \times \ZZ_2$:}
\begin{itemize}
\item[$G$:] $\langle (12), (13), (14), (56)\rangle < {\mathfrak S}_6$
\item[$T_i$:] $(2,2,2,2,2)$ and $(2,4,6)$
\item[$S_1$:] (12)(34)(56), (34)(56), (13), (23), (13)
\item[$S_2$:] (13)(56), (1342), (124)(56)
\item[$H_1$:] $\ZZ_2^3$
\item[$\pi_1$:] $\pi_1= \ZZ^2 \rtimes_\varphi \ZZ_2$, 
$\varphi(\overline{1})=-Id$
\item[$\HH$:] $\HH_1 := c_3d_3^2d_2^{-1}$, $\HH_2 := c_1c_4c_2d_3^2$, 
$\HH_3 := c_1c_4d_1d_3^2$,
$\HH_4 := c_2c_3c_1d_3^2$
\item[] $\HH_1 \mapsto \overline{1}$, $\HH_2$, $\HH_3 \mapsto(1,0)$, 
$\HH_4 \mapsto (0,1)$.
\end{itemize}

\subsection{$K^2=4$, Galois group $\ZZ_2^4 \rtimes_\varphi \ZZ_2$:} 
$\varphi(\overline{1})=
\begin{pmatrix}
\overline{1}&0\\
\overline{1}&\overline{1}
\end{pmatrix}
\oplus
\begin{pmatrix}
\overline{1}&0\\
\overline{1}&\overline{1}
\end{pmatrix}
$

\begin{itemize}
\item[$G$:] $\langle (12)(34), (14)(23), (56)(78), (58)(67), (13)(57) 
\rangle< {\mathfrak S}_8$
\item[$T_i$:] $(2,2,2,4)$, and $(2,2,2,4)$
\item[$S_1$:] (24)(68), (12)(34)(56)(78), (12)(34), (24)(5876)
\item[$S_2$:] (12)(34)(57)(68), (56)(78), (24)(68), (1234)(5876)
\item[$H_1$:] $\ZZ_4^2$
\item[$\pi_1$:] $\langle 
x_1,x_2,y|x_i^4,y^2,[x_i,y],x_1^{-1}x_2^{-1}x_1x_2y\rangle$ of order 
$32$
\item[$\HH$:] $\HH_1 := c_4c_3d_4$, $\HH_2 := c_4d_1d_4^{-1}$, $\HH_3 
:= c_1c_4^{-1}d_1d_3d_4$
\item[] $\HH_1 \mapsto x_1$, $\HH_2 \mapsto x_2$, $\HH_3 \mapsto x_2y$
\end{itemize}

\subsection{$K^2=4$, Galois group ${\mathfrak S}_4$:}

\begin{itemize}
\item[$G$:] ${\mathfrak S}_4$
\item[$T_i$:] $(2,2,2,2,2)$ and $(3,4,4)$
\item[$S_1$:] (13), (14), (12)(34), (12), (14)
\item[$S_2$:] (132), (1432), (1342)
\item[$H_1$:] $\ZZ_2^2 \times \ZZ_4$
\item[$\pi_1$:] $\ZZ^2 \rtimes_\varphi \ZZ_4$, $\varphi(\overline{1})=-Id$
\item[$\HH$:] $\HH_1 := c_5c_2$, $\HH_2 := c_2c_3d_3$, $\HH_3 :=
c_1c_2d_2d_3^{-1}$, $\HH_4 := c_2c_1c_2d_1^{-1}d_2$
\item[] $\HH_1 \mapsto (1,0)$, $\HH_2 \mapsto (0,-1)\cdot \overline{1}$, $\HH_3
\mapsto (0,-1) \cdot \overline{2}$, and $\HH_4 \mapsto \overline{1}$
\end{itemize}

\subsection{$K^2=4$, Galois group ${\mathfrak S}_3 \times \ZZ_3$:}

\begin{itemize}
\item[$G$:] $\langle (12), (13), (456) \rangle < {\mathfrak S}_6$
\item[$T_i$:] $(2,2,3,3)$ and $(3,6,6)$
\item[$S_1$:] (13), (13), (123)(456), (132)(465)
\item[$S_2$:] (132), (23)(465), (12)(456)
\item[$H_1$:] $\ZZ_3^2$
\item[$\pi_1$:] $\ZZ^2 \rtimes_\varphi \ZZ_3$, $\varphi(\overline{1})=
\begin{pmatrix}
0&-1\\
1&-1
\end{pmatrix}$
\item[$\HH$:] $\HH_1 := c_2c_1$, $\HH_2 := c_3c_1d_3$, $\HH_3 := c_4c_1d_2$,
$\HH_4 := c_4c_1d_3^{-1}d_1$
\item[] $\HH_1\mapsto Id$, $\HH_2 \mapsto \overline{1}$, $\HH_3 \mapsto
(1,0)\cdot \overline{1}$, $\HH_4 \mapsto (-1,0)$
\end{itemize}

\subsection{$K^2=4$, Galois group $\ZZ_3^2 \rtimes_\varphi \ZZ_2$:}
$\varphi(\overline{1})=-Id$

\begin{itemize}
\item[$G$] $\langle (123), (456), (12)(45) \rangle < {\mathfrak S}_6$
\item[$T_i$:] $(2,2,3,3)$ and $(2,2,3,3)$
\item[$S_1$:] (23)(56), (23)(45), (123)(465), (132)
\item[$S_2$:] (23)(56), (12)(46), (132)(465), (465)
\item[$H_1$:] $\ZZ_3^3$
\item[$\pi_1$:] $\ZZ_3^3$
\item[$\HH$:]
$\HH_1 := c_1c_2d_4^{-1}$,
$\HH_2 := c_2d_1d_4^{-1}$,
$\HH_3 := c_3c_2c_1d_3^{-1}$,
$\HH_4 := c_3c_4^{-1}d_3$
\item[] $\HH_1$, $\HH_2 \mapsto (\overline{1},\overline{0},\overline{0})$,
$\HH_3 \mapsto (\overline{0},\overline{1},\overline{0})$,  $H_4 \mapsto
(\overline{0},\overline{0},\overline{1})$
\end{itemize}

\subsection{$K^2=4$, Galois group $D_4 \times \ZZ_2$:}\label{thestrangeone}

\begin{itemize}
\item[$G$:] $\langle (1234), (12)(34), (56) \rangle \rangle < {\mathfrak S}_6$
\item[$T_i$:] $(2,2,2,2,2)$ and $(2,2,2,4)$
\item[$S_1$:] (56), (24), (12)(34), (12)(34)(56), (24)
\item[$S_2$:] (13)(56), (14)(23)(56), (13)(24)(56), (1234)(56)
\item[$H_1$:] $\ZZ_2^2 \times \ZZ_4$
\item[$\pi_1$:] $\langle x_1,x_2,x_3|x_3^2,x_2^4,
     x_1^2 x_2^2,
     (x_1^{-1} x_3 x_2)^2,
     (x_2^{-1} x_1^{-1} x_3)^2,
     x_3 x_2 x_1 x_3 x_2^{-1} x_1^{-1}\rangle$
\item[$\HH$:] $\HH_1 := c_1c_3c_4$, $\HH_2 := c_2d_1d_3$,
$\HH_3 := c_3c_2c_1d_4^{-1}$, $\HH_4 := c_3d_1d_4^{-1}$
\item[]
$\HH \rightarrow \pi_1$ maps $H_1 \mapsto x_1x_3^{-1}x_2x_3^{-1}$,
$H_2 \mapsto x_1$,  $H_3 \mapsto x_2$,  $H_4 \mapsto x_3$
\end{itemize}

The normal subgroups of $\pi_1$ of minimal index
among the normal subgroups with free abelianization are $4$, all
isomorphic to $\ZZ^2$ and of index $8$. The quotient of $\pi_1$ by
them is either $D_4$ (in two cases) or $\ZZ_2 \times \ZZ_4$ (in two
cases). This in particular shows that this group is different from all
fundamental groups of surfaces with $p_g=q=0$ and $K^2=4$ we know.

\subsection{$K^2=4$, Galois group $\ZZ_4 \times \ZZ_2$}\label{thestrangetwo}

\begin{itemize}
\item[$G$:] $\ZZ_4 \times \ZZ_2$
\item[$T_i$:] $(2,2,4,4)$ and $(2,2,4,4)$
\item[$S_1$:] $(\overline{2},\overline{1})$, $(\overline{2},\overline{1})$,
$(\overline{3},\overline{1})$, $(\overline{1},\overline{1})$
\item[$S_2$:] $(\overline{0},\overline{1})$, $(\overline{0},\overline{1})$,
$(\overline{3},\overline{0})$, $(\overline{1},\overline{0})$
\item[$H_1$:] $\ZZ_2^3 \times \ZZ_4$
\item[$\pi_1$:] $\langle 
x_1,x_2,x_3,x_4|(x_3x_4)^2,(x_3^{-1}x_4)^2,x_2x_4^{-1}x_2x_4, 
[x_1,x_2], [x_2,x_3],$
\item[] $[x_1,x_4], x_1^{-1}x_3x_1^{-1}x_3^{-1},
x_4^{-1}x_3^{-1}x_1x_3x_4x_1,
x_2x_4^{-1}x_3^{-1}x_4^{-1}x_2^{-1}x_3^{-1},$
\item[] $x_1^{-1}x_4x_1^{-1}x_3^{-1}x_4x_2x_1^{-2}x_3^{-1}x_2^{-1} \rangle$
\item[$\HH$:] $\HH_1 := c_2c_1$,
$\HH_2 := d_2d_1$,
$\HH_3 := c_3d_1d_4^{-1}$,
$\HH_4 := c_3c_1d_4$,
$\HH_5 := c_1d_4^2d_1$
\item[] $\forall i \leq 4:\ \HH_i \mapsto x_i$, $\HH_5 \mapsto x_3x_4$
\end{itemize}

There is an exact sequence
$$\begin{matrix}
1 \ra & \ZZ^4 &\ra &\pi_1 &\ra \ZZ_2^2 \ra 1\\
& e_1 &\mapsto &x_1&\\
& e_2 &\mapsto &x_2&\\
& e_3 &\mapsto &x_3^2&\\
& e_4 &\mapsto &x_4^2&
\end{matrix}.$$

The image of $\ZZ^4$ in $\pi_1$ is the only normal subgroup of index $\leq 4$
with free abelianization.

\subsection{$K^2=4$, Galois group $\ZZ_2^3$}\label{thestrangethree}

\begin{itemize}
\item[$G$:] $\ZZ_2^3$
\item[$T_i$:]$(2,2,2,2,2)$ and $(2,2,2,2,2)$
\item[$S_1$:]$(\overline{0},\overline{0},\overline{1})$, 
$(\overline{0},\overline{1},\overline{1})$, 
$(\overline{0},\overline{0},\overline{1})$, 
$(\overline{1},\overline{1},\overline{1})$, 
$(\overline{1},\overline{0},\overline{0})$
\item[$S_2$:] $(\overline{1},\overline{0},\overline{0})$, 
$(\overline{1},\overline{0},\overline{1})$, 
$(\overline{0},\overline{1},\overline{0})$, 
$(\overline{1},\overline{1},\overline{0})$, 
$(\overline{1},\overline{0},\overline{1})$
\item[$H_1$:] $\ZZ_2^3 \times \ZZ_4$
\item[$\pi_1$:] $\langle x_1,x_2,x_3,x_4|x_2^2,(x_2x_4)^2, 
(x_2x_1^2)^2, (x_2x_3^2)^2, [x_1,x_3], x_4x_3^{-1}x_4x_3,$
\item[] $x_1^{-1}x_4x_1x_4, x_2x_1x_3x_2x_3x_1, 
x_3^{-1}x_1x_2x_1x_3^{-1}x_2\rangle$
\item[$\HH$:] $\HH_1 := c_1c_2d_3$,
$\HH_2 := c_1c_5d_2$,
$\HH_3 := c_1d_1d_2$,
$\HH_4 := c_2c_3d_3$,
$\HH_5 := c_5d_3d_4$
\item[] $\HH_1 \mapsto x_1$,
$\HH_2 \mapsto x_2$,  $\HH_3 \mapsto x_2^{-1}$,  $\HH_4 \mapsto x_3$, 
$\HH_5 \mapsto x_4$
\end{itemize}

There is an exact sequence
$$\begin{matrix}
1 \ra & \ZZ^4 &\ra &\pi_1 &\ra \ZZ_2^2 \ra 1\\
& e_1 &\mapsto &x_3x_1&\\
& e_2 &\mapsto &x_4&\\
& e_3 &\mapsto &x_1^2&\\
& e_4 &\mapsto &(x_2x_1)^2&
\end{matrix}.$$

The image of $\ZZ^4$ in $\pi_1$ is the only normal subgroup of index $\leq 4$
with free abelianization.

This group is isomorphic to the one in \ref{thestrangetwo}.

\subsection{$K^2=6$, Galois group ${\mathfrak A}_6$:}

\begin{itemize}
\item[$G$:] ${\mathfrak A}_6$
\item[$T_i$:] $(2,5,5)$ and $(3,3,4)$
\item[$S_1$:] (16)(34), (25436), (16452)
\item[$S_2$:] (156), (146)(235), (1532)(46)
\item[$H_1$:] $\ZZ_{15}$
\item[$\pi_1$:] ${\mathfrak A}_4 \times \ZZ_5$
\item[$\HH$:] $\HH_1 := c_2c_1d_2d_3^{-1}$, $\HH_2 := 
c_2c_3^{-1}c_1c_2^{-2}d_3^{-1}d_1d_3^{-1}$
\item[] $\HH_1\mapsto (234)\cdot \overline{1}$, $\HH_2 \mapsto (123)$
\end{itemize}

\subsection{$K^2=6$, Galois group ${\mathfrak A}_6$:}

\begin{itemize}
\item[$G$:] ${\mathfrak A}_6$
\item[$T_i$:] $(2,5,5)$ and $(3,3,4)$
\item[$S_1$:] (16)(34), (25436), (16452)
\item[$S_2$:] (123)(456), (125), (1465)(23)
\item[$H_1$:] $\ZZ_{15}$
\item[$\pi_1$:] ${\mathfrak A}_4 \times \ZZ_5$
\item[$\HH$:] $\HH_1 := c_2^{-2}d_1d_2^{-1}$, $\HH_2 := 
c_2c_3^{-1}c_1d_1d_3^{-2}$
\item[] $\HH_1 \mapsto (234)\cdot \overline{1}$,
$\HH_2 \mapsto (123)\cdot \overline{4}$
\end{itemize}

\subsection{$K^2=6$, Galois group ${\mathfrak S}_5 \times \ZZ_2$:}

\begin{itemize}
\item[$G$:] $\langle (12), (13), (14), (15), (67) \rangle < {\mathfrak S}_7$
\item[$T_i$:] $(2,4,6)$ and $(2,4,10)$
\item[$S_1$:] (13)(45)(67), (1524)(67), (153)(24)
\item[$S_2$:] (25)(67), (1432), (15234)(67)
\item[$H_1$:] $\ZZ_2 \times \ZZ_4$
\item[$\pi_1$:] ${\mathfrak S}_3 \times D_{4,5,-1}$, recall that $D_{4,5,-1}=
\langle x,y|x^4,y^5,xyx^{-1}y \rangle$
\item[$\HH$:] $\HH_1 := c_2c_1d_2d_1d_2^{-1}d_3^{-1}$,  $\HH_2 :=
c_2c_3^{-2}d_2^{-1}d_3$, $\HH_3 := c_2^2c_1c_2^{-1}d_2^2d_3^{-1}d_1$
\item[] $\HH_1 \mapsto (123)xy$, $\HH_2 \mapsto (12)x$,  $\HH_3 
\mapsto (132)xy^3$
\end{itemize}

\subsection{$K^2=6$, Galois group $PSL(2,7)$:}
\begin{itemize}
\item[$G$:] $\langle (34)(56),(123)(457) \rangle < {\mathfrak S}_7$
\item[$T_i$:] $(2,7,7)$ and $(3,3,4)$
\item[$S_1$:] (27)(46), (1235674), (1653742)
\item[$S_2$:] (157)(234), (145)(367), (2476)(35)
\item[$H_1$:] $\ZZ_{21}$
\item[$\pi_1$:] ${\mathfrak A}_4 \times \ZZ_7$
\item[$\HH$:] $\HH_1 := c_2^3d_3^{-1}d_1$,  $\HH_2 := 
c_3^3d_1d_3^{-1}$, $\HH_3 := c_3^2c_2^{-1}d_1d_3$
\item[] $H_1 \mapsto (123)\cdot \overline{1}$, $H_2$, $H_3 \mapsto 
(134) \cdot \overline{6}$
\end{itemize}

\subsection{$K^2=6$, Galois group $PSL(2,7)$:}
\begin{itemize}
\item[$G$:] $\langle (34)(56),(123)(457) \rangle < {\mathfrak S}_7$
\item[$T_i$:] $(2,7,7)$ and $(3,3,4)$
\item[$S_1$:] (27)(46), (1235674), (1653742)
\item[$S_2$:] (127)(364), (157)(234), (1263)(57)
\item[$H_1$:] $\ZZ_{21}$
\item[$\pi_1$:] ${\mathfrak A}_4 \times \ZZ_7$
\item[$\HH$:] $\HH_1 := c_3c_2^{-1}d_2d_1^{-1}$, $\HH_2 := 
c_1c_3^2d_1^{-1}d_3$, $\HH_3 := c_2^2c_1d_1d_3d_1^{-1}$
\item[] $\HH_1 \mapsto (134) \cdot \overline{1}$,
$\HH_2 \mapsto (132) \cdot \overline{1}$, $\HH_3 \mapsto (12)(34) 
\cdot \overline{6}$
\end{itemize}

\subsection{$K^2=6$, Galois group ${\mathfrak A}_5$:}
\begin{itemize}
\item[$G$:] ${\mathfrak A}_5$
\item[$T_i$:] $(2,3,3,3)$ and $(2,5,5)$
\item[$S_1$:] (13)(24), (123), (235), (254)
\item[$S_2$:] (23)(45),(15342), (13425)
\item[$H_1$:] $\ZZ_3 \times \ZZ_{15}$
\item[$\pi_1$:] $\ZZ^2 \rtimes_\varphi \ZZ_{15}$, $\varphi(\overline{1})=
\begin{pmatrix}
0&-1\\
1&-1
\end{pmatrix}$
\item[$\HH$:] $\HH_1= c_1d_1d_3d_2d_3^{-1}$, $\HH_2 = 
c_1c_4^{-1}c_2^{-1}d_2d_3^{-1}$, $\HH_3 = c_1c_2^{-1}c_1d_1d_3^2$
\item[] $\HH_1 \mapsto \overline{12}$,
$\HH_2 \mapsto (1,0) \cdot \overline{14}$, $\HH_3 \mapsto \overline{1}$
\end{itemize}

\subsection{$K^2=6$, Galois group ${\mathfrak S}_4 \times \ZZ_2$:}
\begin{itemize}
\item[$G$:] $\langle  (12), (13), (14), (56) \rangle < {\mathfrak S}_6$
\item[$T_i$:] $(2,2,2,2,4)$ and $(2,4,6)$
\item[$S_1$:] (23)(56), (12)(34)(56), (13)(56), (13)(56), (1342)
\item[$S_2$:] (13), (1324)(56), (142)(56)
\item[$H_1$:] $\ZZ_2^3 \times \ZZ_4$
\item[$\pi_1$:] $\langle x_1,x_2,x_3,x_4| x_4^4, x_4^{-2}x_3^2, 
x_3^{-1}x_2^{-1}x_3x_2^{-1}$,
\item[] $x_1^{-1} x_3 x_4^{-1} x_1^{-1} x_3^{-1} x_4,
     x_2 x_3 x_1^{-1} x_2 x_3^{-1} x_1^{-1},
     x_1 x_4^{-1} x_3^{-1} x_1 x_4 x_3\rangle$
\item[$\HH$:]  $\HH_1= c_4c_3$, $\HH_2 = c_3c_2c_1d_3$,
$\HH_3 = c_5d_2d_3^{-1}$, $\HH_4 = c_1d_2^{-1}d_3^2$
\item[]  $\HH_i \mapsto x_i$
\end{itemize}

There is an exact sequence
$$\begin{matrix}
1 \ra & \Pi_2 &\ra &\pi_1 &\ra \ZZ_2 \times \ZZ_4 \ra 1\\
& \alpha_1 &\mapsto &x_3x_4x_2^{-1}x_4^{-1}x_3^{-1}&\\
& \beta_1 &\mapsto &x_1[x_4^{-1},x_3^{-1}]&\\
& \alpha_2 &\mapsto &x_1&\\
& \beta_2 &\mapsto &x_2^{-1}&
\end{matrix}.$$

The image of $\Pi_2$ in $\pi_1$ is of minimal index
among the  normal subgroup of $\pi_1$ with free
abelianization.

\subsection{$K^2=6$, Galois group $D_4 \times \ZZ_2$:}

\begin{itemize}
\item[$G$:] $\langle (1234), (12)(34), (56) \rangle < {\mathfrak S}_6$
\item[$T_i$:] $(2,2,2,2,4)$ and $(2,2,2,4)$
\item[$S_1$:] (56), (56), (12)(34)(56), (13)(56), (1432)
\item[$S_2$:] (24), (14)(23), (13)(24)(56), (1432)(56)
\item[$H_1$:] $\ZZ_2^2 \times \ZZ_4^2$
\item[$\pi_1$:] $\langle x_1,x_2,x_3,x_4| x_4^2, (x_1^{-1}x_4)^2, 
(x_3^2x_4)^2, (x_4x_2^2)^2$,
\item[] $[x_1,x_2^2], [x_1,x_3^2], [x_2,x_3^2],  [x_3,x_2^2], 
x_1^{-1}x_2x_3^{-1}x_1^{-1}x_2^{-1}x_3 \rangle$
\item[$\HH$:] $\HH_1= c_2c_1$, $\HH_2 = c_3d_1d_4^{-1}$,
$\HH_3 = c_4d_1d_3$, $\HH_4 = c_1d_4^2d_3$
\item[] $\HH_i \mapsto x_i$
\end{itemize}

There is an exact sequence
$$\begin{matrix}
1 \ra & \ZZ^2 \times \Pi_2 &\ra &\pi_1 &\ra \ZZ_2^2 \ra 1\\
& ((1,0),id) &\mapsto &x_2^2&\\
& ((0,1),id) &\mapsto &x_3^2&\\
& ((0,0),\alpha_1)&\mapsto &x_1x_2&\\
& ((0,0),\beta_1) &\mapsto &x_3^{-1}x_1^{-1}&\\
& ((0,0),\alpha_2)&\mapsto &x_4x_2x_4x_1^{-1}&\\
& ((0,0),\beta_2) &\mapsto &x_1x_4^{-1}x_3^{-1}x_4^{-1}&
\end{matrix}.$$

The image of $\ZZ^2 \times \Pi_2$ in $\pi_1$ is the only normal 
subgroup of index $\leq 4$ with free abelianization.

%%%%%%%%%%%%%%%%%%%%%%%%%%%%%%%%%%%%%%%%%%%%%%%%%%%%%%

\bigskip
\noindent
{\bf Authors Adresses:}

\noindent
Ingrid Bauer, Fabrizio Catanese \\
Lehrstuhl Mathematik VIII, Mathematisches Institut der Universit\"at 
Bayreuth;\\
Universit\"atsstr. 30;
D-95447 Bayreuth, Germany\\
Fritz Grunewald \\
Mathematisches Institut der Heinrich-Heine-Universit\"at D\"usseldorf;\\
Universit\"atsstr. 1;
D-40225 D\"usseldorf, Germany \\
Roberto  Pignatelli \\
Dipartimento di Matematica della Universit\`a di Trento;\\
Via Sommarive 14;
I-38123 Trento (TN), Italy\\

\appendix
\section*{Appendix: MAGMA scripts} {\small
\begin{verbatim}

/* The first script, ListOfTypes, produces, for each value of K^2=8-t,
a list containing all signatures fulfilling the conditions i), ii) and iii)
in subsection 5.2. We represent the signatures by sequences of 7
positive integers [m1,..,m7], adding some "1" if the length of the
sequence is shorter. E.g., the signature (2,3,7) is represented by the
sequence [1,1,1,1,2,3,7]. Note that by lemma 5.9 the length of the
signatures is bounden by 7.

The script takes all nondecreasing sequences of 7 positive integers
between 1 and  3(K^2+2) (cf. lemma 5.9) and discards those
which do not fulfill all three conditions.
The script uses two auxiliary functions Theta and Alpha computing the
respective numbers (see (17) in subsection 5.2). */

Theta:=function(sig)
   a:=5;
   for m in sig do a:=a-1/m; end for;
   return a;
end function;

Alpha:=func<sig,Ksquare | Ksquare/(4*Theta(sig))>;

ListOfTypes:=function(Ksquare: MaximalCoefficient:=3*(Ksquare+2))
      list:=[* *];
      for m1 in [ 1..3*(Ksquare+2)] do
      for m2 in [m1..3*(Ksquare+2)] do
      for m3 in [m2..3*(Ksquare+2)] do
      for m4 in [m3..3*(Ksquare+2)] do
      for m5 in [m4..3*(Ksquare+2)] do
      for m6 in [m5..3*(Ksquare+2)] do
      for m7 in [m6..3*(Ksquare+2)] do
        sig:=[Integers() | m1,m2,m3,m4,m5,m6,m7 ];
        if Theta(sig) ne 0 then A:=Alpha(sig,Ksquare);
        if A in IntegerRing() and A ge 1 then
        if forall{m : m in sig | 2*A/m in IntegerRing()} then bads:=0;
           for m in sig do
             if A/m notin IntegerRing() then bads +:=1;
           end if; end for;
        if bads le 4-Ksquare/2 then Append(~list,sig);
        end if; end if; end if; end if;
      end for; end for; end for; end for; end for; end for; end for;
return list;
end function;

/* The second important script is ListGroups, which returns for each
K^2, the list of all triples [G,T1,T2] where T1 and T2 are two
signatures in ListOfTypes(K^2) and G is a group of the prescribed
order 8*Alpha(T1)*Alpha(T2)/K^2 having sets of spherical generators of
both signatures. Note that the script skips the cases: |G| = 1024,
1152, or bigger than 2000, cases which we have excluded by hand (cf.
subsection(5.3)).

It uses two subscripts:
   ElsOfOrd which returns the set of elements of a group of a given
order,
   ExistSphericalGenerators, a boolean function which answers if a
group has a system of generators of a prescribed signature */

ElsOfOrd:=function(G,order)
      Els:={ };
      for g in G do if Order(g) eq order then Include(~Els, g);
      end if; end for;
      return Els; end function;

ExistSphericalGenerators:=function(G,sig)
      test:=false;
      for x1 in ElsOfOrd(G,sig[1]) do
      for x2 in ElsOfOrd(G,sig[2]) do
      for x3 in ElsOfOrd(G,sig[3]) do
      for x4 in ElsOfOrd(G,sig[4]) do
      for x5 in ElsOfOrd(G,sig[5]) do
      for x6 in ElsOfOrd(G,sig[6]) do
        if Order(x1*x2*x3*x4*x5*x6) eq sig[7]
           and #sub<G|x1,x2,x3,x4,x5,x6> eq #G
           then test:=true; break x1;
        end if;
      end for; end for; end for; end for; end for; end for;
return test;
end function;

ListGroups:=function(Ksquare)
      list:=[* *]; L:=ListOfTypes(Ksquare); L1:=L;
        for T1 in L do
        for T2 in L1 do
          ord:=8*Alpha(T1,Ksquare)*Alpha(T2,Ksquare)/Ksquare;
        if ord in IntegerRing() and ord le 200 and ord notin {1024, 1152} then
        for G in SmallGroups(IntegerRing()!ord: Warning := false) do
        if ExistSphericalGenerators(G,T1) then
        if ExistSphericalGenerators(G,T2) then
          Append(~list,[* G,T1,T2 *]);
        end if; end if; end for; end if; end for;
        L1:=Reverse(Prune(Reverse(L1)));
        end for;
return list;
end function;

/* Each triple [G,T1,T2] corresponds to many surfaces, one for each
pair of spherical generators of G of the prescribed signatures, but
still these surfaces can be too singular.

The script ExistingNodalSurfaces returns all triples in the output of
ListGroups such that at least one of the surfaces has exactly the
prescribed number 8-K^2 of nodes as singularities.
It uses three more scripts.

FSGUpToConjugation returns a list of spherical generators of the group
of given type, and more precisely, one for each conjugacy
class. It divides the set of spherical generators into two sets,
Heaven and Hell, according to the rule "if a conjugate of the set I'm
considering is in Heaven, go to Hell, else to Heaven", and returns Heaven.

CheckSingsEl is a Boolean function answering if the surface S given by
two sets of generators of G has the right number of nodes and no other
singularities.
It uses the following: given a pair of spherical generators, the singular
points of the resulting surface S come from pairs g,h of elements, one
for each set, such that there are n,m with g^n nontrivial and
conjugated to h^m. If the order of g^n is 2 then S has some nodes,
else S has worse singularities, a contradiction.

Checksings is a Boolean function answering if there is apair of spherical
generators of G of signatures sig1 and sig2 giving a surface with the
right number of nodes.
To save time it checks only one set of spherical generators for each
conjugacy class, using FSGUpToConjugation. In fact, the isomorphism
class of the surface obtained by a pair of spherical generators does
not change if we act on one of them by an inner automorphism. */


FSGUpToConjugation:=function(G,sig)
     Heaven:={@ @}; Hell:={@ @};
      for x1 in ElsOfOrd(G,sig[1]) do
      for x2 in ElsOfOrd(G,sig[2]) do
      for x3 in ElsOfOrd(G,sig[3]) do
      for x4 in ElsOfOrd(G,sig[4]) do
      for x5 in ElsOfOrd(G,sig[5]) do
      for x6 in ElsOfOrd(G,sig[6]) do x7:=x1*x2*x3*x4*x5*x6;
        if Order(x7) eq sig[7] then
        if #sub<G|x1,x2,x3,x4,x5,x6> eq #G then
        if [x1,x2,x3,x4,x5,x6,x7^-1] notin Hell then
         Include(~Heaven,[x1,x2,x3,x4,x5,x6,x7^-1]);
          for g in G do
          Include(~Hell, [x1^g,x2^g,x3^g,x4^g,x5^g,x6^g,(x7^-1)^g]);
           end for;
          end if; end if; end if;
        end for; end for; end for; end for; end for; end for;
return Heaven;
end function;

CheckSingsEl:=function(G,seq1,seq2,Ksquare)
      Answer:=true; Nodes:=0;
        for g1 in seq1 do
        for g2 in seq2 do
        for d1 in [1..Order(g1)-1] do
        for d2 in [1..Order(g2)-1] do
        if IsConjugate(G,g1^d1,g2^d2) then
        if Order(g1^d1) ge 3 then Answer:=false; break g1;
        elif Order(g1^d1) eq 2 then
          Nodes +:=Order(G)/(2*d1*d2*#Conjugates(G,g1^d1));
        if Nodes gt 8-Ksquare then Answer:=false; break g1;
        end if; end if; end if; end for; end for; end for; end for;
     return Answer and Nodes eq 8-Ksquare;
end function;

CheckSings:=function(G,sig1,sig2,Ksquare)
      test:=false;
        for gens1 in FSGUpToConjugation(G,sig1) do
        for gens2 in FSGUpToConjugation(G,sig2) do
        if CheckSingsEl(G,gens1,gens2,Ksquare) then test:=true; break gens1;
        end if; end for; end for; return test;
end function;

ExistingNodalSurfaces:=function(Ksquare)
      M:=[* *];
      for triple in ListGroups(Ksquare) do
         G:=triple[1]; T1:=triple[2]; T2:=triple[3];
      if CheckSings(G,T1,T2,Ksquare) then Append(~M, triple);
      end if; end for; return M; end function;

/* ExistingNodalSurfaces produces a list of triples [G,T1,T2] more
precisely 7 for K^2=2, 11 for K^2=4 and 6 for K^2=6, one for each row
of table 2. To each of these triples correspond at least a surface
with p_g=q=0.
We need a script which finds all surfaces, modulo isomorphisms. Recall
that two of these surfaces are isomorphic if and only if the two pairs
of  spherical generators are equivalent for the equivalence
relation generated by Hurwitz moves on each set and by the
automorphism group of G (acting simultaneously on the pair). This is
done by FindAllComponents, which needs 4 new scripts.

AutGr  describes the automorphism group of G as set.

HurwitzMove runs an Hurwitz move on a sequence of elements of a
group.

HurwitzOrbit computes the orbit of a sequence of elements of
the group under Hurwitz moves, and then returns (to save memory)
the subset of the sequences such that the  corresponding sequence of integers,
given by the orders of the group elements, is non
decreasing.

SphGens gives all sets of spherical generators of a group of
prescribed signature of length 5. Note that we have reduced the
length of the signature from 7 to 5, because the output of
ListOfTypes (see table 3) shows that the bound of r in lemma 5.9 can
be sharpened to 5.

Finally FindAllComponents produces (fixing the group and the
signatures) one pair of spherical generators for each isomorphism
class.
Running FindAllComponents on all 7+11+6 triples obtained by
ExixtingNodalSurfaces (remembering that we have to shorten the
signatures removing the first two 1) we always find one surface
except in three cases, giving two surfaces.
*/

AutGr:=function(G)
      Aut:=AutomorphismGroup(G); A:={ Aut!1 };
      repeat
        for g1 in Generators(Aut) do
        for g2 in A do
          Include (~A,g1*g2);
        end for; end for;
      until  #A eq #Aut; return A; end function;

HurwitzMove:=function(seq,idx)
      return Insert(Remove(seq,idx),idx+1,seq[idx]^seq[idx+1]);
end function;

HurwitzOrbit:=function(seq)
      orb:={ }; Purgatory:={ seq };
      repeat
        ExtractRep(~Purgatory,~gens); Include(~orb, gens);
          for k in [1..#seq-1] do hurgens:=HurwitzMove(gens,k);
            if hurgens notin orb then Include(~Purgatory, hurgens);
          end if; end for;
      until IsEmpty(Purgatory);
      orbcut:={  };
      for gens in orb do test:=true;
        for k in [1..#seq-1] do
          if Order(gens[k]) gt Order(gens[k+1]) then test:=false; break k;
          end if;
        end for;
        if test then Include(~orbcut, gens);
        end if;
      end for; return orbcut; end function;

SphGens:=function(G,sig)
      Gens:={ };
      for x1 in ElsOfOrd(G,sig[1]) do
      for x2 in ElsOfOrd(G,sig[2]) do
      for x3 in ElsOfOrd(G,sig[3]) do
      for x4 in ElsOfOrd(G,sig[4]) do
      if Order(x1*x2*x3*x4) eq sig[5] then
      if sub<G|x1,x2,x3,x4> eq G then
        Include(~Gens, [x1,x2,x3,x4,(x1*x2*x3*x4)^-1]);
      end if; end if; end for; end for; end for; end for;
return Gens; end function;

FindAllComponents:=function(G,sig1,sig2,Ksquare)
      Comps:={@ @}; Heaven:={ }; Hell:={ }; Aut:=AutGr(G);
      NumberOfCands:=#SphGens(G,sig1)*#SphGens(G,sig2);
      for gen1 in SphGens(G,sig1) do
      for gen2 in SphGens(G,sig2) do
        if gen1 cat gen2 notin Hell then
          Include(~Heaven, [gen1,gen2]);
          orb1:=HurwitzOrbit(gen1); orb2:=HurwitzOrbit(gen2);
          for g1 in orb1 do for g2 in orb2 do for phi in Aut do
             Include(~Hell, phi(g1 cat g2));
             if #Hell eq NumberOfCands then break gen1;
             end if;
          end for; end for; end for;
        end if;
      end for; end for;
      for gens in Heaven do
      if CheckSingsEl(G,gens[1],gens[2],Ksquare) then
        Include(~Comps, gens);
      end if; end for; return Comps; end function;

/* Finally we have to compute the fundamental group of the resulting
surfaces, which is done by the script Pi1. It uses the script
Polygroup which, given a sequence of 5 spherical generators of a group
5, produces the corresponding Polygonal Group P and the surjective
morphism P->G.

Pi1 uses the two sequences seq1, seq2 of 5 spherical generators of G
to construct the surjection f from the product of the respective
polygonal groups T1 x T2 to G x G, defines the subgroup H (=preimage
of the diagonal in GxG), and takes the quotient by by a sequence of
generators of Tors(H), which are obtained as follows: consider each pair
of elements (g1,g2), where gi is an element of seqi, such that g1^a is
nontrivial and conjugate to g2^b. Note that they give rise to a node, so
they both have order 2). Let h be a fixed element of G such that
g1^a*h=h*g2^b.
We let c vary between the elements in G commuting with g1^a.
Let t1,t2 be the natural preimages of g1,g2 in the
respective polygonal groups T1 and T2, t a preimage of h^-1*c in T2.
Then t1^a*t^-1*t2^b*t has order 2 and belongs to H.
These elements generate Tors(H) (see prop. 4.3). */


PolyGroup:=function(seq)
   F:=FreeGroup(#seq);
   P:=quo<F | F.1^Order(seq[1]), F.2^Order(seq[2]), F.3^Order(seq[3]),
             F.4^Order(seq[4]), F.5^Order(seq[5]), F.1*F.2*F.3*F.4*F.5>;
   return P, hom<P->Parent(seq[1])|seq>;
end function;

Pi1:=function(seq1,seq2)
      T1:=PolyGroup(seq1); T2,f2:=PolyGroup(seq2); G:=Parent(seq1[1]);
      T1xT2:=DirectProduct(T1,T2);
      inT2:=hom< T2->T1xT2 | [T1xT2.6, T1xT2.7, T1xT2.8, T1xT2.9, T1xT2.10]>;
      GxG,inG:=DirectProduct(G,G); m:=NumberOfGenerators(G); L:=[ ];
      for i in [1..m] do Append(~L,GxG.i*GxG.(i+m)); end for;
      Diag:=hom<G->GxG|L>(G);
      f:=hom<T1xT2->GxG|
                      inG[1](seq1[1]),inG[1](seq1[2]),inG[1](seq1[3]),
                                        inG[1](seq1[4]),inG[1](seq1[5]),
                      inG[2](seq2[1]),inG[2](seq2[2]),inG[2](seq2[3]),
                                        inG[2](seq2[4]),inG[2](seq2[5])>;
      H:=Rewrite(T1xT2,Diag@@f); TorsH:=[ ];
      for i in [1..5] do if IsEven(Order(seq1[i])) then
      for j in [1..5] do if IsEven(Order(seq2[j])) then
      a:=IntegerRing()!(Order(seq1[i])/2); b:=IntegerRing()!(Order(seq2[j])/2);
      test,h:= IsConjugate(G,seq1[i]^a,seq2[j]^b);
         if test then for c in Centralizer(G,seq1[i]^a) do
           Append(~TorsH, T1xT2.i^a * ((T1xT2.(j+5)^b)^(inT2((h^-1*c) @@ f2))));
         end for; end if;
        end if; end for; end if; end for;
      return Simplify(quo<H | TorsH>);
end function;

/* This following script does the same computation as the previous one,
but instead of returning the fundamental group as astract group it
returns T1xT2, H as subgroup of T1xT2 and a list of generators of
Tors(H) */

Pi1Detailed:=function(seq1,seq2)
      T1:=PolyGroup(seq1); T2,f2:=PolyGroup(seq2); G:=Parent(seq1[1]);
      T1xT2:=DirectProduct(T1,T2);
      inT2:=hom< T2->T1xT2 | [T1xT2.6, T1xT2.7, T1xT2.8, T1xT2.9, T1xT2.10]>;
      GxG,inG:=DirectProduct(G,G); m:=NumberOfGenerators(G); L:=[ ];
      for i in [1..m] do Append(~L,GxG.i*GxG.(i+m)); end for;
      Diag:=hom<G->GxG|L>(G);
      f:=hom<T1xT2->GxG|
                      inG[1](seq1[1]),inG[1](seq1[2]),inG[1](seq1[3]),
                                        inG[1](seq1[4]),inG[1](seq1[5]),
                      inG[2](seq2[1]),inG[2](seq2[2]),inG[2](seq2[3]),
                                        inG[2](seq2[4]),inG[2](seq2[5])>;
      H:=Rewrite(T1xT2,Diag@@f); TorsH:=[ ];
      for i in [1..5] do if IsEven(Order(seq1[i])) then
      for j in [1..5] do if IsEven(Order(seq2[j])) then
      a:=IntegerRing()!(Order(seq1[i])/2); b:=IntegerRing()!(Order(seq2[j])/2);
      test,h:= IsConjugate(G,seq1[i]^a,seq2[j]^b);
         if test then for c in Centralizer(G,seq1[i]^a) do
           Append(~TorsH, T1xT2.i^a * ((T1xT2.(j+5)^b)^(inT2((h^-1*c) @@ f2))));
         end for; end if;
        end if; end for; end if; end for;
      return T1xT2,H, TorsH;
end function;


\end{verbatim}
}

\begin{thebibliography}{Grif-Schmid}
%%%%%%%%%%%%%%%%%%%%%%%%%%%%%%%%%%%%%%%%%%%%%%%%%%%%%%




\bibitem[Arm65]{armstrong1} Armstrong, M. A., {\it  On the
fundamental group of an
orbit space.} Proc. Cambridge Philos. Soc. {\bf  61} 639--646 (1965) .

\bibitem[Arm68]{armstrong2} Armstrong, M. A., {\it The fundamental
group of the orbit
space of a discontinuous group.} Proc. Cambridge Philos. Soc. {\bf 64
} 299--301 (1968).


\bibitem[Bar84]{barlow2} Barlow, R., {\it Some new surfaces with
$p\sb g=0$}. Duke Math.
J.  \textbf{51},  no. 4, 889--904 (1984).

\bibitem[Bar85a]{barlow}
Barlow, R.,
{\it A simply connected surface of general type with $p\sb g=0$}.
   Invent. Math.  \textbf{79}  (1985),  no. 2, 293--301.

\bibitem[Bar85b]{barlowbloch}
Barlow, R.,  {\em Rational equivalence of zero cycles for some more 
surfaces with $p\sb g=0$.}  Invent. Math.  \textbf{79}  (1985),  no. 
2, 303--308.

\bibitem[BPV84]{bpv}
     Barth, W.,  Peters, C.,   Van de Ven, A., {\em
Compact complex surfaces.}
         Ergebnisse der Mathematik und ihrer Grenzgebiete {\bf (3)}
Springer-Verlag, Berlin, (1984).


\bibitem[BHPV04]{bhpv}
     Barth, W., Hulek, K.,  Peters, C.,   Van de Ven, A., {\em
Compact complex surfaces.}
Second edition. Ergebnisse der Mathematik und ihrer Grenzgebiete {\bf
3. Folge.  4.}
Springer-Verlag, Berlin,  xii+436 pp (2004).

\bibitem[BC04]{bacat} I. Bauer, F. Catanese, {\em Some new
surfaces  with $ p_g = q=
0$},  The Fano Conference, 123--142, Univ. Torino,
Turin, (2004).

\bibitem[BC09a]{keumnaie} I. Bauer, F. Catanese, {\em The moduli 
space of Keum-Naie surfaces}, arxiv:0909.1733 (2009).

\bibitem[BC09b]{burniat1} I. Bauer, F. Catanese, {\em Burniat 
surfaces I: fundamental groups and moduli of primary Burniat 
surfaces}, arxiv:0909.3699 (2009).

\bibitem[BC09c]{burniat2} I. Bauer, F. Catanese, {\em Burniat 
surfaces II: the moduli
spaces of secondary Burniat surfaces}, preliminary version.

\bibitem[BCC09]{toappear} I. Bauer, F. Catanese, M. Chan, {\em Inoue 
surfaces with $K^2=7$}, in preparation.

\bibitem[BCG05]{BCG} I. Bauer, F. Catanese, F. Grunewald,
{\it Beauville surfaces without real structures},
In: Geometric Methods in Algebra and Number Theory, ed. by F. Bogomolov,
Y. Tschinkel, Progress in Math. {\bf 235}, Birkh\"auser (2005), 1--42.

\bibitem[BCG08]{bcg} Bauer, I., Catanese, F., Grunewald, F.: {\it
The classification of surfaces with $p\sb g=q=0$
isogenous to a product of curves. }, Pure Appl. Math. Q. {\bf 4},  no.
2, part 1,
547--586 (2008).

\bibitem[BP09]{bp}
Bauer, I., Pignatelli, R.
{\it The classification of minimal standard isotrivial fibrations with
   $p_g=0$.}
Preliminary version (2009).


\bibitem[Bear83]{beardon} Beardon, A. F. {\em The geometry of discrete groups.}
Graduate
Texts in Mathematics, {\bf 91} Springer-Verlag, New York,  xii+337 pp. (1983).

\bibitem[Beau78]{Beauville} A. Beauville, {\em   Surfaces alg\'ebriques
complexes.} Asterisque {\bf 54}, Soc. Math. France, (1978).

\bibitem[Beau96]{beauville96} A. Beauville, {\em   A Calabi-Yau 
threefold with non-abelian fundamental group.} New trends in 
algebraic geometry (Warwick,1996) Lond. Math. Soc. Lecture Note Ser. 
{\bf 264}, Cambridge Univ. Press, Cambridge (1999).

\bibitem[Blo75]{bloch}
Bloch, S.
{\it $K\sb{2}$ of Artinian $Q$-algebras, with application to
algebraic cycles. }
Comm. Algebra  {\bf 3}  (1975), 405--428.

\bibitem[BKL76]{bkl}
Bloch, S., Kas, A., Lieberman, D. {\em Zero cycles on surfaces with 
$p\sb{g}=0$.}  Compositio Math.  33  (1976), no. 2, 135--145.

\bibitem[BCP97]{magmaref}
Bosma, W., Cannon, J., Playoust, C.
{\it The Magma algebra system. I. The user language}.
J. Symbolic Comput., {\bf 24} (3-4):235-265, 1997.


\bibitem[Bre00]{Bru} Breuer, T.,  {\em Characters and Automorphism
Groups of Compact
Riemann Surfaces}, London Math. Soc. Lecture Note Series {\bf 280},
Cambridge University
Press (2000).

\bibitem[Bur66]{burniat}
Burniat, P. {\it  Sur les surfaces de genre $P\sb{12}>1$. }, Ann.
Mat. Pura Appl.
(4)  {\bf 71}  1966 1--24.

\bibitem[Cam32]{Cam} L. Campedelli, {\em Sopra alcuni piani doppi
notevoli con curve di diramazione
del decimo ordine.} Atti Acad. Naz. Lincei {\bf 15}, (1932), 536--542.


\bibitem[Cat81]{Babbage} Catanese, F., {\it Babbage's conjecture,
contact of surfaces,
symmetric determinantal varieties and applications.} Invent. Math.{
\bf  63} , no. 3,
433--465 (1981).

\bibitem[Cat89]{enr}
Catanese, F.
{ \it Everywhere nonreduced moduli spaces. } Invent. Math.  {\bf 98}    no. 2,
293--310, (1989).

\bibitem[Cat-Kol92]{trento}
Catanese, F., Koll\'ar, J.
{ \it Trento examples 2, }in 'Classification of irregular varieties', 
Springer Lecture Notes in
Mathematics vol. 1515, 136--139, (1992).

\bibitem[Cat99]{sbc}
Catanese, F. { \it Singular bidouble covers and the construction of interesting
algebraic surfaces. } Algebraic geometry: Hirzebruch 70 (Warsaw,
1998),  97--120,
Contemp. Math., 241, Amer. Math. Soc., Providence, RI, (1999).

\bibitem[Cat00]{FabIso}
Catanese, F., {\it Fibred Surfaces, varieties
isogeneous to a
product and related moduli spaces}. Amer. J. Math.  \textbf{122}, no.
1, 1--44 (2000)

\bibitem[Cat03]{modreal}
   Catanese, F., {\it Moduli spaces of surfaces and real
structures.} Ann. of Math.(2) { \bf 158} , no. 2, 577--592 (2003).

\bibitem[Cat03b]{fibred}
   Catanese, F., {\it Fibred K\"ahler and quasi-projective groups.}
Adv. Geom.  Special issue dedicated to Adriano Barlotti.
  suppl., S13--S27, (2003).

\bibitem[CatLB97]{clb}
  Catanese, F.; LeBrun, C.{\it  On the scalar curvature of Einstein manifolds.}
   Math. Res. Lett. {\bf  4} ,  no. 6,
843--854 (1997).


\bibitem[CG94]{cg}
Craighero, P.C.; Gattazzo, R,
{\it Quintic surfaces of $P\sp 3$ having a nonsingular model with
    $q=p\sb g=0$, $P\sb 2\not=0$} .
Rend. Sem. Mat. Univ. Padova {\bf 91}  (1994), 187--198.

\bibitem[DW99]{dw}
Dolgachev, I.; Werner, C.{\it A simply connected numerical Godeaux
    surface with ample canonical class}.
   J. Algebraic Geom.  {\bf 8}  (1999),  no. 4, 737--764.
Erratum ibid,10 (2001),  no. 2, 397.

\bibitem[Enr96]{enr96} Enriques, F., {\em Introduzione alla geometria
sopra le superficie algebriche.}
Memorie della Societa' Italiana delle Scienze (detta "dei XL"),
s.3, to. X ,  (1896), 1--81.

\bibitem[EnrMS]{enrMS} Enriques, F., {\em Memorie scelte di geometria,
vol. I, II, III.}
Zanichelli, Bologna, (1956), 541 pp., (1959), 527 pp., (1966), 456 pp. .

\bibitem[Fuj74]{fujiki}  Fujiki, A. {\it On resolutions of cyclic quotient
singularities.} Publ. Res. Inst. Math. Sci. {\bf 10}, no. 1, 293--328
(1974/75).

\bibitem[Ful84]{fulton} Fulton, W., {\em  Intersection theory.}
Ergebnisse der
Mathematik und ihrer Grenzgebiete {\bf  (3) , 2.}
Springer-Verlag, Berlin, xi+470 pp. (1984).

\bibitem[God34]{godold}
Godeaux, L.,
{\em Les surfaces algÔø‡briques non rationnelles de genres 
arithm\'etique et geom\'etrique nuls.}
Hermann \& Cie., Paris, 33 p. (1934).

\bibitem[God35]{god} Godeaux, L., {\em  Les involutions cycliques
appartenant \'a une surface
alg\'ebrique.} Actual. Sci. Ind., {\bf  270}, Hermann, Paris, (1935).




\bibitem[GJZ08]{GZ} Grunewald, F., Jaikin-Zapirain, A., Zalesskii, P. A.,
{\em Cohomological goodness and the profinite completion of Bianchi
groups.} Duke Math. J.
{\bf 144}, no. 1, 53--72  (2008).

\bibitem[GrMe80]{GM12}
Grunewald, F. J., Mennicke, J. L. {\em Some $3$-manifolds arising 
from ${\rm PSL}\sb{2}(Z[i])$.}  Arch. Math. (Basel)  35  (1980), no. 
3, 275--291

\bibitem[GP03]{gp}
Guletskii, V.; Pedrini, C. {\it Finite-dimensional motives and the
conjectures of Beilinson
and Murre.} Special issue in honor of Hyman Bass on his seventieth
birthday. Part III.
$K$-Theory  {\bf 30}  (2003),  no. 3, 243--263.



\bibitem[Hup67]{Huppert} Huppert, B., {\em Endliche Gruppen.I},
     Die Grundlehren der Mathematischen
Wissenschaften, { \bf Band 134 } Springer-Verlag, Berlin-New York
xii+793 pp.(1967).

\bibitem[InMi79]{inosemik}
Inose, H., Mizukami, M. {\em Rational equivalence of $0$-cycles on 
some surfaces of general type with $p\sb{g}=0$.}  Math. Ann.  244 
(1979), no. 3, 205--217.

\bibitem[Ino94]{inoue}
Inoue, M.
{\it Some new surfaces of general type}.
-Tokyo J. Math. {\bf 17}  (1994),  no. 2, 295--319

\bibitem[Keu88]{keum}
Keum, J.
{\it Some new surfaces of general type with $p_g=0$}.
Unpublished manuscript, 1988.


\bibitem[Kim05]{kimura}
Kimura, S. {\it Chow groups are finite dimensional, in some sense.}
Math. Ann.  {\bf 331}
(2005),  no. 1, 173--201.

\bibitem[Kug75]{kug} Kuga, M., {\em FAFA Note.}
     (1975).

\bibitem[LP07]{lp}
Lee, Y.; Park, J.
{\it A simply connected surface of general type with $p\sb g=0$ and
    $K\sp 2=2$}.  Invent. Math.  {\bf 170}  (2007),  no. 3, 483--505.

   \bibitem[LP09]{lp2}
Lee, Y.; Park, J.
{\it A complex surface of general type with $p_g=0$, $K^2=2$ and
    $H_1=Z/2Z$}. Math. Res. Lett. {\bf 16} (2009), no. 2, 323--330.

    \bibitem[MSG]{MA} MAGMA Database of Small Groups; \\
http://magma.maths.usyd.edu.au/magma/htmlhelp/text404.htm.

\bibitem[MP01a]{mlp01}
Mendes Lopes, M.; Pardini, R.
{\it The bicanonical map of surfaces with $p\sb g=0$ and $K\sp 2\geq
   7$}.
Bull. London Math. Soc.  {\bf 33}  (2001),  no. 3, 265--274.

\bibitem[MP01b]{burniatmp}
Mendes Lopes, M. ; Pardini, R.,
{\it A connected component of the moduli space of surfaces with $p\sb g=0$},
Topology {\bf 40} (2001), no. 5, 977--991.

\bibitem[MP04a]{mlp3}
Mendes Lopes, M.; Pardini, R.
{\it A new family of surfaces with $p\sb g=0$ and $K\sp 2=3$}.
Ann. Sci. \'Ecole Norm. Sup. (4)  {\bf 37}  (2004),  no. 4, 507--531

\bibitem[MP04b]{mlp4b}
Mendes Lopes, M.; Pardini, R.
{\it Surfaces of general type with $p\sb g=0, K\sp 2=6$ and non birational
bicanonical map}.  Math. Ann.  {\bf 329}  (2004),  no. 3, 535--552.

\bibitem[MP07]{mlp1} Mendes Lopes, M., Pardini, R., {\em On the algebraic
fundamental group of surfaces with $K\sp 2\leq3\chi$.}  J.
Differential Geom.  {\bf 77}
(2007),  no. 2, 189--199.

\bibitem[MP08]{mlp} Mendes Lopes, M., Pardini, R., {\em  Numerical Campedelli
surfaces with fundamental group of order 9.} J. Eur. Math. Soc.
(JEMS) {\bf 10} , no. 2,
457--476 (2008).

\bibitem[MPR08]{mlpr} Mendes Lopes, M., Pardini, R., Reid, M. {\em  Campedelli
surfaces with fundamental group of order 8.} arXiv : 0805.0006.

\bibitem[Miy76]{miyaokagod}
Miyaoka, Y.
{\it Tricanonical maps of numerical Godeaux surfaces}.
Invent. Math.  34  (1976), no. 2, 99--111

\bibitem[Miy77]{miyaoka}
Miyaoka, Y.
{\it On numerical Campedelli surfaces}.
In: Complex analysis and algebraic geometry, pp. 113--118. Iwanami
Shoten, Tokyo, 1977.

\bibitem[Nai94]{naie94}
Naie, D.
{\it Surfaces d'Enriques et une construction de surfaces de type g\'en\'eral
avec $p\sb g=0$}.
Math. Z.  {\bf 215}  (1994),  no. 2, 269--280.

\bibitem[Nai99]{naie}
Naie, D.
{\em Numerical Campedelli surfaces cannot have the symmetric group as the
algebraic fundamental group}, J. London Math. Soc. (2) {\bf 59} (1999),
no. 3, 813--827.

\bibitem[NP09]{np}
Neves, J.; Papadakis, S.A.
{\it A construction of numerical Campedelli surfaces with torsion
    $\ZZ/6$}.
Trans. Amer. Math. Soc. {\bf 361}  (2009), no. 9, 4999-5021.

\bibitem[OP81]{op}
Oort, F.; Peters, C.
{\it A Campedelli surface with torsion group $Z/2$.}
Nederl. Akad. Wetensch. Indag. Math.  {\bf 43}  (1981), no. 4, 399--407.


\bibitem[PPS07]{pps3}
Park, H.; Park, J.; Shin, D.
{\it A simply connected surface of general type with $p_g=0$ and
    $K^2=3$}. arXiv:0708.0273

\bibitem[PPS08a]{pps3b}
Park, H.; Park, J.; Shin, D.
{\it A complex surface of general type with $p_g=0$,
    $K^2=3$ and $H_1=Z/2Z$ }. arXiv:0803.1322

\bibitem[PPS08b]{pps4}
Park, H.; Park, J.; Shin, D.
{\it A simply connected surface of general type with $p_g=0$ and
    $K^2=4$}. arXiv:0803.3667

\bibitem[Pen09]{penego}
Penegini, M.
{\it The classification of isotrivial fibred surfaces with $p_g=q=2$}.
  arXiv:0904.1352


\bibitem[Pet76]{peterscamp}
Peters, C. A. M.
{\it On two types of surfaces of general type with vanishing geometric genus}.
Invent. Math. {\bf 32} (1976), no. 1, 33--47.

\bibitem[Pet77]{peters}
Peters, C. A. M. {\it On certain examples of surfaces with
$p\sb{g}=0$ due to Burniat. } Nagoya Math. J.  {\bf 66}  (1977), 109--119.

\bibitem[PY07]{p-y}
Prasad, G.; Yeung, S., {\it  Fake projective planes.}  Invent. Math.
{\bf 168}  (2007),  no. 2, 321--370.

\bibitem[PY09]{p-yadd}
Prasad, G.; Yeung, S., {\it  Addendum to ``Fake projective planes''.}
arXiv:0906.4932


\bibitem[Pol07]{pol} Polizzi, F. {\it Standard isotrivial fibrations with
$p_g=q=1$}.
J. Algebra {\bf 321} (2009), no. 6, 1600-1631.

\bibitem[Pol09]{pol2} Polizzi, F.
{\it Numerical properties of isotrivial fibrations}.
Preprint math/0810.4195.

\bibitem[Rei]{MilesCamp} Reid, M.
{\it Surfaces with $p_g = 0$, $K^2 = 2$}.
Preprint available at http://www.warwick.ac.uk/~masda/surf/

\bibitem[Rei78]{tokyo} Reid, M.
{\it Surfaces with $p_g = 0$, $K^2 = 1$}.
J. Fac. Sci. Tokyo Univ. {\bf 25} (1978), 75--92

\bibitem[Rei79]{milesLNM} Reid, M.: {\it $\pi \sb{1}$ for surfaces
with small $K\sp{2}$}.
In: Algebraic geometry (Proc. Summer Meeting, Univ. Copenhagen,
Copenhagen, 1978),
534--544. Lecture Notes in Math. {\bf 732}, Springer, Berlin (1979).

\bibitem[Rei87]{reid} Reid, M.: {\it Young person's guide to
canonical singularities.}
Algebraic geometry, Bowdoin, 1985 (Brunswick, Maine, 1985), 345--414,
Proc. Sympos. Pure
Math., { \bf 46, Part 1}, Amer.
               Math. Soc., Providence, RI, (1987).

\bibitem[Rei91]{cvg}
Reid, M., {\it Campedelli versus Godeaux.} in '  Problems in the
theory of surfaces and
their classification' (Cortona, 1988),  309--365, Sympos. Math.,
XXXII, Academic Press,
London, 1991.

\bibitem[Se90]{Segal} Segal, D. {\em Decidable properties of
polycyclic groups.} Proc.
London Math. Soc. (3) 61 (1990), no. 3, 497--528.

\bibitem[Serra96]{serrano} Serrano, Fernando {\it Isotrivial fibred
surfaces}.   Ann. Mat.
Pura Appl. (4) { \bf  171 } 63--81 (1996).

\bibitem[Serre94]{Serre} Serre, J. P. {\em  Cohomologie galoisienne. }
     Fifth edition. Lecture Notes in Mathematics, 5. Springer-Verlag, Berlin,
     x+181 pp. (1994).

\bibitem[Sha78]{shav} Shavel, I. H., {\em A class of algebraic
surfaces of general type
constructed from quaternion algebras.} Pacific J. Math. {\bf 76},
(1978), no. 1, 221--245.

\bibitem[Sup98]{sup} Supino, P.,
{\em A note on Campedelli surfaces}
Geom. Dedicata  {\bf 71}  (1998),  no. 1, 19--31.

\bibitem[Tole93]{Toledo}Toledo, D.,
{\em  Projective varieties with non-residually finite fundamental group. }
Inst. Hautes Ôø‡tudes Sci. Publ. Math. {\bf No. 77 } (1993), 103--119.

\bibitem[Voi92]{voisin}
Voisin, C. {\em Sur les z\'ero-cycles de certaines hypersurfaces 
munies d'un automorphisme.}  Ann. Scuola Norm. Sup. Pisa Cl. Sci. (4) 
19  (1992),  no. 4, 473--492.

\bibitem[Wer94]{werner}
Werner, C. , {\em A surface of general type with $p\sb
g=q=0$,
$K\sp 2=1$.}  Manuscripta Math.  {\bf 84}  (1994),  no. 3-4, 327--341.

\bibitem[Wer97]{wer97}
Werner, C.
{\it A four-dimensional deformation of a numerical Godeaux surface}.
Trans. Amer. Math. Soc.  {\bf 349}  (1997),  no. 4, 1515--1525.

\bibitem[Xia85]{xiao}
Xiao, Gang
{\it Surfaces fibr\'ees en courbes de genre deux. (French)}
[Surfaces fibered by curves of genus two] Lecture Notes in
Mathematics, 1137. Springer-Verlag, Berlin, 1985.


\bibitem[Yau77]{yau} Yau, S.T., {\em Calabi's conjecture and some new results
in algebraic geometry.} Proc. Nat. Acad.
Sci. U.S.A. {\bf 74} (1977), no. 5, 1798--1799.




\end{thebibliography}
\end{document}